\documentclass[10pt,twocolumn,twoside]{IEEEtran}

\ifCLASSINFOpdf

\else
 
\fi

\usepackage[utf8]{inputenc}
\usepackage{amsmath}
\usepackage{amssymb}
\usepackage{amsthm}
\usepackage[svgnames]{xcolor}
\usepackage{amsfonts}
\usepackage{footnote}
 \usepackage{indentfirst}
\usepackage{enumerate}
\usepackage{graphicx}
\usepackage{array} 
\usepackage{adjustbox}
\usepackage{hyperref}
\hypersetup{colorlinks=true, linkcolor=blue, breaklinks=true, urlcolor=blue}
\usepackage{array,colortbl}
\usepackage{multicol,gensymb,multirow}
\usepackage{threeparttable,sidecap}
\usepackage{enumitem} 
\usepackage[font=small]{caption}
\usepackage{subfigure}
\usepackage{mathtools}
\usepackage{cite}

\newenvironment{reduce}
 {\hbox\bgroup\scriptsize$\displaystyle}
 {$\egroup}
\newcommand{\suchthat}{\;\ifnum\currentgrouptype=16 \middle\fi|\;}

\newcommand{\subscr}[2]{#1_{\textup{#2}}}

\newcommand{\setdef}[2]{\{#1 \; | \; #2\}}

\newcommand{\imagunit}{\mathrm{i}}
\newcommand*\mc[0]{\mathcal}  
\newcommand{\diag}{\mathrm{diag}} 

\newcommand{\colvec}[2][.8]{%
  \scalebox{#1}{%
    \renewcommand{\arraystretch}{.8}%
    $\begin{bmatrix}#2\end{bmatrix}$%
  }
}

\newcommand\oprocendsymbol{\hbox{$\square$}}
\newcommand\oprocend{\relax\ifmmode\else\unskip\hfill\fi\oprocendsymbol}

\renewcommand{\theenumi}{(\roman{enumi})}

\DeclareSymbolFont{bbold}{U}{bbold}{m}{n}
\DeclareSymbolFontAlphabet{\mathbbold}{bbold}
\newcommand{\vect}[1]{\mathbbold{#1}}

\newtheorem{theorem}{Theorem}

\newtheorem{lemma}[theorem]{Lemma}
\newtheorem{corollary}[theorem]{Corollary}
\newtheorem{definition}[theorem]{Definition}
\newtheorem{example}[theorem]{Example}

\newtheorem{remark}[theorem]{Remark}
\usepackage{scalerel}
\usepackage{tikz}
\usetikzlibrary{svg.path}
\definecolor{orcidlogocol}{HTML}{A6CE39}
\tikzset{
	orcidlogo/.pic={
		\fill[orcidlogocol] svg{M256,128c0,70.7-57.3,128-128,128C57.3,256,0,198.7,0,128C0,57.3,57.3,0,128,0C198.7,0,256,57.3,256,128z};
		\fill[white] svg{M86.3,186.2H70.9V79.1h15.4v48.4V186.2z}
		svg{M108.9,79.1h41.6c39.6,0,57,28.3,57,53.6c0,27.5-21.5,53.6-56.8,53.6h-41.8V79.1z M124.3,172.4h24.5c34.9,0,42.9-26.5,42.9-39.7c0-21.5-13.7-39.7-43.7-39.7h-23.7V172.4z}
		svg{M88.7,56.8c0,5.5-4.5,10.1-10.1,10.1c-5.6,0-10.1-4.6-10.1-10.1c0-5.6,4.5-10.1,10.1-10.1C84.2,46.7,88.7,51.3,88.7,56.8z};
	}
}
\newcommand\orcidicon[1]{\href{https://orcid.org/#1}{\mbox{\scalerel*{
				\begin{tikzpicture}[yscale=-1,transform shape]
				\pic{orcidlogo};
				\end{tikzpicture}
			}{|}}}}

\allowdisplaybreaks
\begin{document}

\title{Singular Perturbation and Small-signal Stability\\ for Inverter Networks}

\author{Saber~Jafarpour,~\IEEEmembership{Member,~IEEE,}
        Victor Purba,~\IEEEmembership{Student Member,~IEEE,} Brian B. Johnson,~\IEEEmembership{Member,~IEEE,} \\
        Sairaj V. Dhople,
        ~\IEEEmembership{Member,~IEEE,} and~Francesco Bullo,~\IEEEmembership{Fellow,~IEEE}
\thanks{S. Jafarpour and F. Bullo are with the Center of Control,
  Dynamical Systems and Computation, University of California Santa
  Barbara, CA 93106, USA. E-mail: (saber.jafarpour, bullo@engineering.ucsb.edu).}
\thanks{V. Purba and S. V. Dhople are with Department of Electrical and Computer Engineering at the
  University of Minnesota, Minneapolis, MN 55414 USA. E-mail: (purba002,  dhople@umn.edu).}
\thanks{Brian B. Johnson is with the Department of Electrical Engineering at University of Washington,
Seattle, Washington, WA 98195 USA. E-mail: (brianbj@uw.edu) }}




\markboth{IEEE Transactions on Control of Network Systems, Submitted}%
{Jafarpour \MakeLowercase{\textit{et al.}}}
%



\maketitle

\begin{abstract}
This paper examines small-signal stability of electrical networks composed dominantly of three-phase grid-following inverters. We show that the mere existence of a high-voltage power flow solution does not necessarily imply small-signal stability; this motivates us to develop a framework for stability analysis that systematically acknowledges inverter dynamics. We identify a suitable time-scale decomposition for the inverter dynamics, and using singular perturbation theory, obtain an analytic sufficient condition to verify small-signal stability. Compared to the alternative of performing an eigenvalue analysis of the full-order network dynamics, our analytic sufficient condition reduces computational complexity and yields insights on the role of network topology and constitution as well as inverter-filter and control parameters on small-signal stability. Numerical simulations for a radial network validate the approach and illustrate the efficiency of our analytic conditions for designing and monitoring grid-tied inverter networks. 
\end{abstract}


%
\IEEEpeerreviewmaketitle

\section{Introduction}

\paragraph*{Problem description and motivation} The ongoing shift from fossil-fuel-driven synchronous generators to
power-electronics-interfaced renewable energy is leading to changes in how
power grids are modeled, analyzed, and controlled. While synchronous
generators are generally rated at several hundreds of MVA and installed on
the transmission backbone, power electronics inverters are distributed
across both transmission and distribution subsystems and are generally much
smaller in capacity. Furthermore, synchronous generators have large rotating
masses that buffer supply-demand fluctuations and limit frequency
excursions during transients, whereas inverters have very
different dynamics (attributable dominantly to output filters and digital
controllers~\cite{AY-RI:10}) and they possess no moving parts. Future grids will have large number of power-electronics-interfaced generations with highly distributed architecture as inverters
assume a more prominent role, and this will necessitate the
development of compatible models and computationally efficient analysis approaches to certify
stability.

Most commercial inverters on the market for residential and utility-scale applications are \emph{grid-following}. This means inverters inject currents while synchronized to the voltage at their terminals which is assumed to be set externally by the bulk grid. In effect, grid-following inverters act as voltage-following current sources. In recent years, there has been increased attention on \emph{grid-forming} inverter technology, whereby---much like conventional synchronous generators---terminal voltage and frequency are modulated as a function of real- and reactive-power injections. 
Indeed, while grid-forming technology may very well be the solution for grids with penetration levels approaching $100\%$, in the foreseeable future, it is likely that the bulk grid will see conventional synchronous generators co-exist alongside a large number of grid-following inverters. For instance, in Oahu, Hawaii, at least $800,000$ micro-inverters interconnect photovoltaic panels to the grid, producing as much power as the state’s largest conventional power plant, Konkar~\cite{PF:15}. This motivates the problem of stability analysis for large-scale networks composed dominantly of grid-following inverters that we examine in this work.

The inverter dynamical model that we investigate is composed of a current
controller, a power controller, a PLL, and an LCL filter. This is
prototypical and mirrors models published widely in the
literature~\cite{CP-MG-TG-RI:10,NP-MP-TCG:07,MP-TG:03,VP-SVD-SJ-FB-BBJ:17p,VP-SJ-BBJ-FB-SVD:16l,MR-JAM-JWK:14,MR-JAM-JWK:15,ET-DH:03}. We
propose a framework that leverages singular perturbation methods to obtain
an analytic, computationally light-weight condition for small-signal
stability assessment of three-phase distribution networks with
grid-following inverters. Our solution strategy yields an analytic
condition for stability that is agnostic to system size and clearly
highlights the role of the network topology and pertinent system parameters
on system stability. Due to the complexity of the involved dynamics, most
prior art on stability of inverter-based systems has typically focused on
simplified models that neglect inner control loops which underpin fast
dynamics~\cite{JS-DZ-RO-AMS-TS-JR:16} (see the discussions
in~\cite{SYC-PT:15}). Some exceptions
are~\cite{JLA-MB-JL-LM:11,CY-XZ-FL-HX-RC-HN:17}, where stability of
full-order inverter models are studied in grid-connected networks. However,
these studies are restricted to parallel networks of inverters and are not
applicable to networks with general topologies. In some cases, detailed
inverter models have been considered in general networks, but system
stability has only been studied for a single inverter or a small network of
inverters using numerical eigenvalue analysis. For instance, small-signal
stability is analyzed in the literature using eigenvalue analysis for a
single grid-following inverter under unintentional islanding
in~\cite{DV-VJ:20}, for the IEEE $37$-bus system with $7$ inverters
in~\cite{MR-JAM-JWK:14,MR-JAM-JWK:15}, for a radial network consisting of
$3$ inverters in~\cite{NP-MP-TCG:07}, and for a
single-machine-single-inverter network in~\cite{YL-BJ-VG-VP-SD:17}.
Understandably, while numerical eigenvalue analysis of the linearized
network is indeed a reasonable strategy, this approach comes with
significant drawbacks. First, the large size of the network combined with
the high dimensionality of the inverter model pose computational challenges
to analysis. Furthermore, studying small-signal stability by computing
eigenvalues of the linearized system does not reveal the role of critical
network attributes on system stability, insights, which if formalized
appropriately can facilitate analysis and design.


On a tangential note, it must be acknowledged
that there is a growing body of work on stability assessment of synchronous generators and grid-forming inverter systems, and this includes approaches that have applied model-order reduction using singular perturbation analysis. For instance, in~\cite{SC-DG-FD:17,DG-CA-FD:18}, a model-reduction approach based on singular perturbation is proposed to study stability and control of grid-forming inverters. In~\cite{LL-SVD:14}, singular perturbation is applied to obtain a hierarchy of reduced-order models for inverters in the grid-forming mode. In~\cite{MR-JAM-JWK:15}, a suitable time-scale decomposition for
a class of inverters is identified and an iterative scheme for model order reduction is proposed. We refer interested readers to~\cite{VRS-JOR-PVK:84} for a survey on singular perturbation methods and to~\cite{JS-DZ-RO-AMS-TS-JR:16} for a survey on
application of singular perturbation in stability and control
of inverters.


\paragraph*{Contributions} We make several contributions to the study of small-signal stability
of grid-following inverter networks. First, we show that adopting a
static model for grid-following inverters (as fixed sources
of active and reactive power) and neglecting fast dynamics induced by
the inverters' control loops may lead to erroneous conclusions regarding stability (see Example~\ref{exm:dynamic_vs_static}). This underscores the importance of acknowledging a
full-order model for stability analysis. We start by introducing a  model for grid-tied inverter networks that acknowledges line dynamics and where each inverter is modeled using 
  a $13$th-order model. Next, we
uncover a correspondence between the equilibrium points of the
 dynamics and the solutions of the algebraic
power-flow equations. The main contribution of this paper, i.e., an analytic sufficient condition for small-signal stability, is derived in the context of a dimensionless transcription of the involved models which leads to the identification of a physically insightful parametrization of the inverters. We show that certain assumptions on the range of parameters result in a time-scale decomposition of the system. Using singular perturbation analysis, we propose an analytic sufficient condition which guarantees small-signal stability over a given parametric regime. As a unique contribution, we emphasize that the dimensionless form of the network equations as well
as the regularity of the singular perturbation problem (i.e., existence of isolated quasi-steady state manifolds) are critical
steps in a rigorous time-scale analysis. Over this specified parametric regime, our analytic sufficient condition can also be interpreted as a lower bound on the stability threshold of the network and allows us to check system stability with minimal computational complexity. Furthermore, in the special case of resistive networks, our sufficient condition reduces to checking Hurwitzness of a Metzler matrix, something that can be implemented efficiently via linear programming (see~\cite{AR:15}).

In the literature, small-signal stability of systems is usually studied using eigenvalue analysis for the full-order models
  (e.g., see~\cite{NP-MP-TCG:07}) or for the reduced-order models
  (e.g., see the survey~\cite{JS-DZ-RO-AMS-TS-JR:16}). Compared to performing eigenvalue analysis for the full-order system, our proposed analytically driven sufficient condition reduces computational complexity by addressing the high dimensionality of the underlying dynamics (the inverter model we study has $13$ dynamical states) and it demarcates the role of the network (topology and constitution) and pertinent inverter dynamics (filter and controller parameters). Moreover, while many existing approaches in the literature only outline iterative schemes for model reduction (see~\cite{MR-JAM-JWK:15}), our analysis provides
  an explicit reduced-order model. Finally, we provide several numerical case studies that validate the analysis.


\paragraph*{Paper organization} In Section~\ref{sec:Model of a decoupled inverter}, we present the
dynamical model for a class of three-phase grid-following
inverters. In Section~\ref{sec:Equilibrium points of a grid-tied
  network of inverters}, we derive an equivalent dimensionless description
for a grid-tied network
of inverters and loads and we study the equilibrium points of the
system. In Section~\ref{sec:stability_analysis}, we provide a sufficient condition for existence of a locally exponentially stable equilibrium point. Finally, in Section~\ref{sec:numerical}, we illustrate some applications of the theoretical results in network design and stability assessment.

\paragraph*{Notation} \emph{Vectors and matrices.} We denote the set of real numbers by $\mathbb{R}$, the set of complex
numbers by $\mathbb{C}$, the set of complex numbers with
negative real part by ${\mathbb{C}}_{-}$, the set of binary $n$-tuples
by ${\mathbb{Z}}^n_2$, and the $n$-dimensional torus by $\mathbb{T}^n$. We define $\imagunit=\sqrt{-1}$. We
identify the complex plane $\mathbb{C}$ with the real plane $\mathbb{R}^2$. For
a complex number $v = v_1 + \imagunit v_2\in \mathbb{C}$, the norm of
$v$ is $|v| = \sqrt{v_1^2+v_2^2}$ and the argument of $v$,
$\mathrm{arg}(v)$, is the angle between $v$ and the positive imaginary
axis. We denote the identity matrix of dimension $n$ by $I_n$, the $n$-column vector
of zeros with $\vect{0}_n$, and the $n$-column vector of ones with $\vect{1}_n$. For a matrix $A=\{a_{ij}\}\in
\mathbb{C}^{n\times m}$, we denote the trace by $\mathrm{tr}(A)$, the determinant by $\det(A)$, and $\infty$-norm by $\|A\|_{\infty} = \max_{i}\sum_{j=1}^{n} |a_{ij}|$. 
For two real symmetric matrices $A,B\in \mathbb{R}^{n\times n}$
we write $A\succ B$ if $A-B$ is positive definite. A
  real square matrix $A\in \mathbb{R}^{n\times n}$ is Metzler if all its
  off-diagonal entries are non-negative. For two square
matrices $A\in \mathbb{R}^{n\times n}$ and $B\in \mathbb{R}^{m\times m}$, the
tensor product is denoted by $A\otimes B$. For a vector
$\mathrm{x}\in \mathbb{C}^n$, we denote $\diag(\mathrm{x})$ by
$[\mathrm{x}]$. 

\newcommand{\mJ}{\mathcal{J}}
\newcommand{\mH}{\mathcal{H}}
\newcommand{\mR}{\mathcal{R}}
\newcommand{\mD}{\mathcal{D}}

\noindent \emph{From $n$-complex variables to $2n$-real variables.}
For every complex $Z = X +\imagunit Y \in \mathbb{C}$,
the associated real variable in the real plane is denoted by $z = (x,
  y)^{\top} \in \mathbb{R}^2$. We will frequently use matrix $\mJ=\begin{pmatrix}0 & -1\\ 1&
0\end{pmatrix}$, matrix $\mH=\begin{pmatrix}0 & 1\\ 1& 0\end{pmatrix}$, and the rotation
matrix (parameterized by angle $\theta \in \mathbb{S}^1$) by
$\mR(\theta)=\begin{pmatrix}\cos(\theta) &
\sin(\theta)\\ -\sin(\theta) & \cos(\theta)\end{pmatrix}$. Let
$\mathrm{u}\in \mathbb{R}^{2}$. 
We define matrix-valued operators $\mD:\mathbb{R}^{2}\to \mathbb{R}^{2\times 2}$ and
$\mD':\mathbb{R}^{2}\to \mathbb{R}^{2\times 2}$ by 
\begin{align*}
\mD(\mathrm{u})=\begin{pmatrix}u_1
&u_2\\ u_2& -u_1\end{pmatrix}, \quad \mD'(\mathrm{u})=\begin{pmatrix}u_1
&u_2\\ -u_2& u_1\end{pmatrix}.
\end{align*}
All the above matrices and operators can be extended to the
$n$-dimensional complex and $2n$-dimensional real spaces using block
diagonal structure. For the brevity of notation, we denote the extended
$n$-dimensional complex ($2n$-dimensional real) matrix/operator with the
same symbol as its $1$-dimensional complex ($2$-dimensional real)
counterpart. Let $\mathbf{V}\in \mathbb{C}^n$ and
$\mathbf{v}\in \mathbb{R}^{2n}$ be the associated real vector, then we
define $\|\mathbf{v}\|_{\mathbb{C},\infty} = \|\mathbf{V}\|_{\infty}$.

\noindent \emph{Algebraic graph theory.} We denote an undirected weighted graph by a triple
$G=(\mathcal{N},\mathcal{E},A)$, where $\mathcal{N}=\{1,2,\ldots,n\}$ is the set of
nodes and $\mathcal{E}\subseteq \mathcal{N}\times \mathcal{N}$ is the set of edges. The
matrix $A=\{a_{ij}\}\in \mathbb{R}^{n\times n}$ is the weighted adjacency
matrix. For every node $i\in V$, 
the degree of the node $i$ is given by $d_i=\sum_{j=1}^{n} a_{ij}$. For a fixed
orientation on $G$, the incidence matrix of the graph $G$ is
denoted by $B\in \mathbb{R}^{n\times m}$. The Laplacian for the graph $G$
is defined by $L=D-A$, where $D=\mathrm{diag}(d_1,d_2,\ldots,d_n)$.

\noindent \emph{Power systems.}
We consider two different reference frames. The first is the so-called \textit{global} $\mathrm{DQ}$-frame, and it is a  rotating reference frame tied to the nominal grid frequency $\subscr{\omega}{nom}$. The second frame, which is usually
referred as the \textit{local} $\mathrm{dq}$-frame, is with reference to each inverter's terminal voltage vector. 
For a balanced three-phase signal $x:\mathbb{R}_{\ge  0} \to \mathbb{R}^3$, we denote the $\mathrm{dq}$-frame representation by
$x_{\mathrm{dq}}=(x_{\mathrm{d}},x_{\mathrm{q}})^{\top}$ and the $\mathrm{DQ}$-frame
representation by
$x_{\mathrm{DQ}}=(x_{\mathrm{D}},x_{\mathrm{Q}})^{\top}$. These  are related as follows:
$x_{\mathrm{dq}}=\mR(\delta) x_{\mathrm{DQ}}$, where $\delta$
is the angle between the $\mathrm{dq}$-frame and the
$\mathrm{DQ}$-frame. We assume that all the electrical
  signals in the network are balanced~\cite[Chapter 2]{HA-AM:17}. Therefore, the voltage of the
  grid is a three-phase AC signal given by the time-varying function $\subscr{v}{g}(t)=
  [\subscr{v}{ga}(t), \subscr{v}{gb}(t),
  \subscr{v}{gc}(t)]^{\top}$ and in the
  global $\mathrm{DQ}$-frame, it is represented as $
    \subscr{v}{gDQ} = [0 ,\subscr{V}{g}]^{\top}$,
  where $\subscr{V}{g}$ is the amplitude of the grid voltage.

\section{Model of Individual Inverter}\label{sec:Model of a decoupled
  inverter} We briefly overview the dynamics of the type of grid-following
$3$-phase inverters examined in this work. For a more detailed description
of the model, see~\cite{VP-SJ-BBJ-FB-SVD:16l,MR-JAM-JWK:14}. The model
captures all relevant AC-side dynamics, and is composed of a:
i)~phase-locked loop (PLL), ii)~power controller, iii)~current controller,
and iv)~$LC$ output filter. An illustrative block diagram is given in
Fig.~\ref{fig:inverter_3phase}.
\begin{figure*}[!t]
\centering
\includegraphics[width=0.8\textwidth]{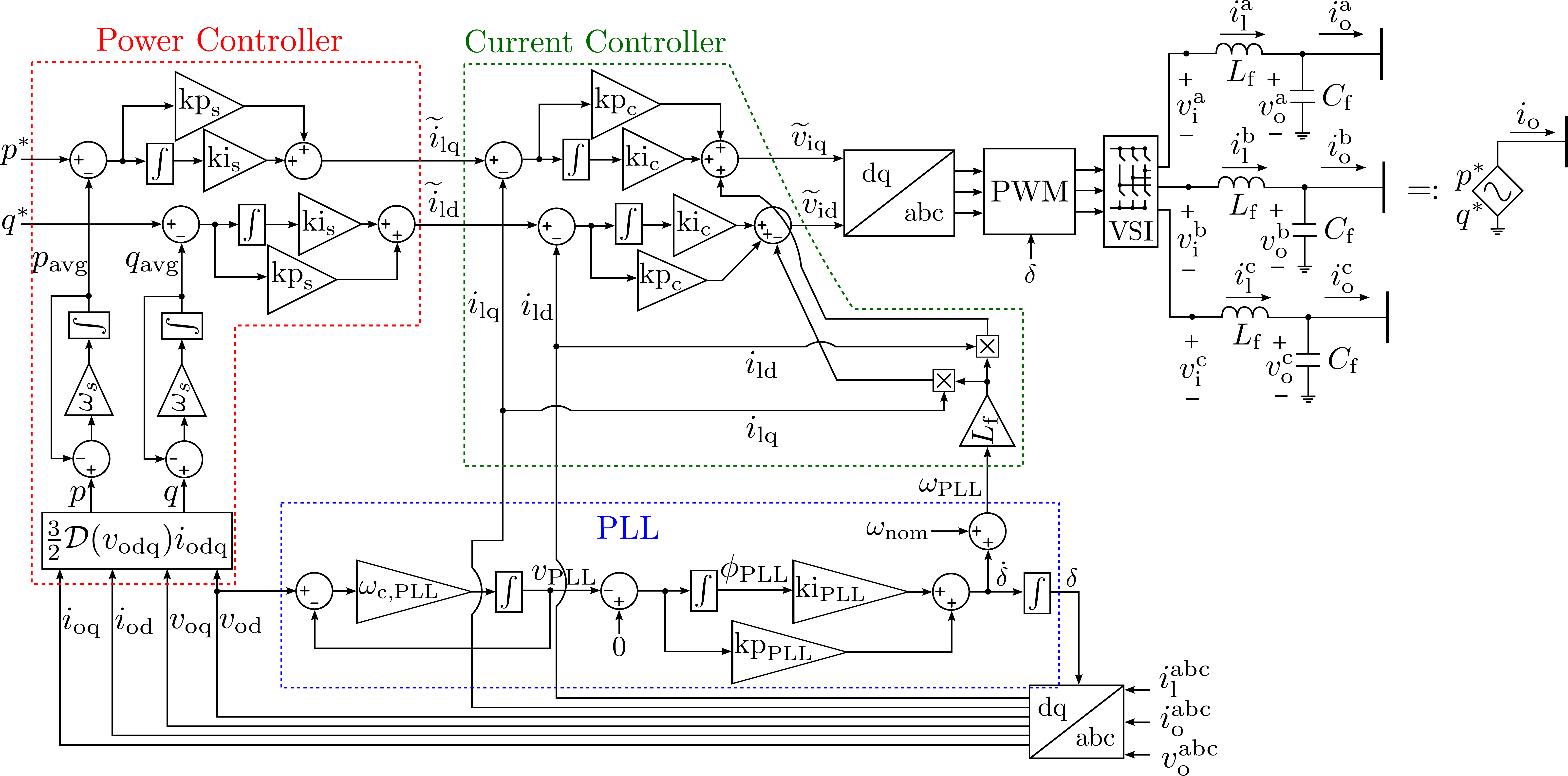}
\caption{Block diagram capturing the dynamics of a three phase grid-tied inverter and shorthand representation for the per-phase equivalent circuit. Model includes dynamics arising from the phase locked loop~\eqref{eq:PLL}, power controller~\eqref{eq:PC}, current controller~\eqref{eq:CC}, and $LC$ filter~\eqref{eq:LC}~\cite{VP-BBJ-SJ-FB-SVD:18h}.}
\label{fig:inverter_3phase}
\end{figure*}
The PLL consists of a low-pass filter with cut-off frequency $\omega_{c,\mathrm{PLL}}$ and a PI controller with gains $\mathrm{kp}_\mathrm{PLL}$ and $\mathrm{ki}_\mathrm{PLL}$. The PLL dynamics are 
\begin{subequations}\label{eq:PLL}
\begin{align}
\dot{v}_{\mathrm{PLL}} &= \omega_{c,\mathrm{PLL}} ( v_{\mathrm{od}} -
                         v_\mathrm{PLL} ), \label{eq:PLL1} \\
\dot{\phi}_{\mathrm{PLL}} & = - v_\mathrm{PLL}, \label{eq:PLL2} \\
\dot{\delta} &= - \mathrm{kp}_\mathrm{PLL} v_\mathrm{PLL} + \mathrm{ki}_\mathrm{PLL} \phi_\mathrm{PLL}, \label{eq:PLL3}
\end{align}
\end{subequations}
where $v_\mathrm{PLL}$ and $\phi_\mathrm{PLL}$ denote states of the
low-pass filter and PI controller, respectively, $\dot{\delta}$ is the
output of the PI controller, and $\subscr{v}{od}$ is the $d$-component
of the output voltage of the inverter. The frequency of the PLL loop
is defined by $\omega_{\mathrm{PLL}}=\omega_{\mathrm{nom}}+\dot
\delta$. The power controller consists of two low-pass filters with cut-off frequency $\subscr{\omega}{s}$ and two PI controllers with gains $\subscr{\mathrm{kp}}{s}$ and $\subscr{\mathrm{ki}}{s}$. The pertinent dynamics are given by:
\begin{subequations} \label{eq:PC}
\begin{align}
\hspace{-0.3cm}\dot{s}_{\mathrm{avg}} &= \subscr{\omega}{s}(s - s_{\mathrm{avg}}), \label{eq:PC1}\\ 
\hspace{-0.3cm}\tilde{i}_{l\mathrm{dq}}
                       &=\subscr{\mathrm{kp}}{s}\mH(s^{\mathrm{ref}}-s_{\mathrm{avg}})+\subscr{\mathrm{ki}}{s}\int\mH(s^{\mathrm{ref}}-s_{\mathrm{avg}})
                         dt, \label{eq:PC2}
\end{align}
\end{subequations}
where $s_\mathrm{avg} = [p_\mathrm{avg}, q_\mathrm{avg}]^\top$
collects the states of the low-pass filters,
$\tilde{i}_{l\mathrm{dq}}$ capture the outputs of the PI controllers
(these are the references for the current controller), $s^{\mathrm{ref}}=(p^{\mathrm{ref}},
q^{\mathrm{ref}})^{\top}$ collects the active- and reactive-power references, and $s=(p,q)^{\top}$ collects instantaneous
active- and reactive-power outputs (measured at the point of common
coupling):
\begin{equation*}
s=\tfrac{3}{2}\begin{pmatrix}
v_{\mathrm{od}} i_{\mathrm{od}} + v_{\mathrm{oq}}
i_{\mathrm{oq}}\\
v_{\mathrm{oq}} i_{\mathrm{od}}
- v_{\mathrm{od}} i_{\mathrm{oq}}
\end{pmatrix} = \tfrac{3}{2}\mD(v_{\mathrm{odq}})i_{\mathrm{odq}}. 
\end{equation*}
The current controller consists of two PI controllers with gains
$\subscr{\mathrm{kp}}{c}$ and $\subscr{\mathrm{ki}}{c}$, with outputs to be the references for the inverter voltage at the switching terminals $v_{\mathrm{idq}}$:
\begin{align}\label{eq:CC}
\begin{split}
\tilde{v}_{\mathrm{idq}}&=
\subscr{\mathrm{kp}}{c}\left(\tilde{i}_{\mathrm{ldq}}-i_{\mathrm{ldq}}\right)+\subscr{\mathrm{ki}}{c}\int\left(\tilde{i}_{\mathrm{ldq}}-i_{\mathrm{ldq}}\right)
dt\\
&+\omega_{\mathrm{PLL}}\subscr{L}{f}(\mJ \subscr{i}{ldq}). 
\end{split}
\end{align}
Since the switching period is typically much shorter than the filter
and controller time constants, we assume that
$v_{\mathrm{idq}}=\tilde{v}_{\mathrm{idq}}$. The dynamics of the $LC$ filter are given by
\begin{subequations} \label{eq:LC}
\begin{align}
\dot{i}_{\mathrm{ldq}} &= \frac{1}{\subscr{L}{f}} ( 
                         v_{\mathrm{idq}} - v_{\mathrm{odq}} ) -
                         \omega_\mathrm{PLL} (\mJ i_{\mathrm{ldq}}), \label{eq:LC1} \\
\dot{v}_{\mathrm{odq}} & =\frac{1}{\subscr{C}{f}} ( i_{\mathrm{ldq}} - i_{\mathrm{odq}} ) - \omega_\mathrm{PLL} (\mJ v_{\mathrm{odq}}). \label{eq:LC2}
\end{align}
\end{subequations}
Finally, we introduce two new variables $\subscr{\phi}{s} := \int (s^{\mathrm{ref}}-\subscr{s}{avg}) dt$, and $\subscr{\gamma}{dq} := \int (\subscr{\tilde{i}}{ldq}
-\subscr{i}{ldq}) dt$, that will aid in exposition. The inner- and outer-loop control architecture examined here is
ubiquitous, see, e.g.,~\cite{CP-MG-TG-RI:10,NP-MP-TCG:07,MP-TG:03,VP-SVD-SJ-FB-BBJ:17p,VP-SJ-BBJ-FB-SVD:16l,MR-JAM-JWK:14,MR-JAM-JWK:15,ET-DH:03}, where similar models are utilized. 

Now, we write the dynamical system of the inverters in the global $\mathrm{DQ}$-frame. To this end, we introduce:
\begin{equation*}
\gamma_{\mathrm{DQ}} = \mR(-\delta)\gamma_{\mathrm{dq}}, \,\,
i_{\mathrm{lDQ}}=\mR(-\delta) {i}_{\mathrm{ldq}},\,\, {v}_{\mathrm{oDQ}}=\mR(-\delta) {v}_{\mathrm{odq}}.
\end{equation*}
For vector $y$ with time-varying entries, the time derivatives in the global $\mathrm{DQ}$-frame and local $\mathrm{dq}$-frame are
related by
\begin{equation*} 
\subscr{\dot{y}}{DQ} - \mJ \mR(-\delta) \dot{\delta}\subscr{y}{DQ} = \mR(-\delta)\subscr{\dot{y}}{dq}.
\end{equation*}
Leveraging this identity, the dynamical model for the grid-following inverter in the $\mathrm{DQ}$-frame can be expressed as:
\begin{align}
\begin{split}\label{eq:inverter_dynamics}
\dot{y}=f(y)+g(y)\subscr{i}{oDQ}+Cs^{\mathrm{ref}},
\end{split}
\end{align}
where $y=(v_\mathrm{PLL},  \delta, \subscr{\phi}{s},  \subscr{s}{avg},
\gamma_{\mathrm{DQ}}, {i}_{\mathrm{lDQ}},
{v}_{\mathrm{oDQ}})^{\top}\in \mathbb{R}^{13}$ captures states of the inverter
in the global $DQ$ reference frame, $s^{\mathrm{ref}}=(p^{\mathrm{ref}},q^{\mathrm{ref}})^{\top}\in \mathbb{R}^2$
captures the references for active and reactive power, $f:\mathbb{R}^{13}\to \mathbb{R}^{13}$ is the drift vector field, and $g:\mathbb{R}^{13}\to \mathbb{R}^{13\times 2}$, and $C\in \mathbb{R}^{13\times 2}$ are control vector fields. The mappings $f,g$ and  the matrix $C$ are obtained from the dynamics governing the current controller, the power controller, the PLL, and the $LC$ filter outlined previously.

\section{Model for Grid-tied Network of Inverters}\label{sec:Equilibrium points of a grid-tied network of inverters}
In this section, we derive the dynamical system model governing the
grid-tied network of inverters and loads and study the equilibrium
points of the system. We model the network using an undirected,
connected, complex-weighted graph $G$ with node set (buses)
$\mc N$, edge set (branches) $\mc E \subseteq \mc N\times \mc N$, and
the symmetric matrix-valued edge weights (admittances)
$a_{kj}=a_{jk}=(R_{kj}I_2 + \subscr{\omega}{nom}
L_{kj}\mJ)^{-1}$, where $R_{kj}$ is the resistance and $L_{kj}$ is the inductance of the line $(k,j)$, for every $(k,j)\in \mathcal{E}$. Suppose $B$
is the incidence matrix of $G$. There are
three types of nodes in the network: we have one \emph{grid
    bus} with voltage $\subscr{v}{gDQ} =
  [0,\subscr{V}{g}]^{\top}$ denoted by $0$.
We have $n\ge 1$ \emph{inverter buses} collected in the set
$\subscr{\mc N}{I}$, and $\ell$ \emph{load buses} collected in the set
$\subscr{\mc N}{L}$. Without loss of generality, we assume that
$\subscr{\mc N}{L} = \{1,\ldots,\ell\}$ and $ \subscr{\mc
  N}{I}=\{\ell+1,\ldots,n+\ell\}$ such that $\mc N =\{0\}\cup \subscr{\mc
  N}{L}\cup \subscr{\mc N}{I}$. Therefore $|\mc N|=1+\ell+n$ and we assume that $|\mc E|=m$. Associated to the matrix-weighted
graph, $G$, we define the nodal admittance matrix by $Y = (B\otimes
I_2)\mathcal{A} (B\otimes I_2)^{\top} \in \mathbb{R}^{2(1+\ell+n)\times 2(1+\ell+n)}$, where $\mathcal{A}\in \mathbb{R}^{2m\times 2m}$ is given by $\mathcal{A} = \mathrm{blkd}(a_{jk})$. The partition $\mc N = \{0\}\cup
\subscr{\mc N}{L} \cup \subscr{\mc N}{I}$ induces the following
decomposition of incidence matrix $B$: $B^{\top}=\begin{pmatrix} \subscr{B^{\top}}{0} &\subscr{B^{\top}}{L} & \subscr{B^{\top}}{I}
\end{pmatrix}$, where $\subscr{B}{0}\in \mathbb{R}^{1\times m}$, $\subscr{B}{L}\in
\mathbb{R}^{\ell\times m}$, and $\subscr{B}{I}\in
\mathbb{R}^{n\times m}$, and the following partition for the admittance matrix $Y$: 
\begin{align*}
Y = \colvec[1]{
\subscr{Y}{00} & \subscr{Y}{0L} & \subscr{Y}{0I}\\
\subscr{Y}{L0} & \subscr{Y}{LL} & \subscr{Y}{LI}\\
\subscr{Y}{I0} & \subscr{Y}{IL} & \subscr{Y}{II}}.
\end{align*}
We also establish the following convention: for a given variable (parameter) $y$ corresponding to the inverter, we
define vector $\mathbf{y}= (y_1^{\top}, \ldots, y_n^{\top})^{\top}$, where
$y_k$ is the associated variable (parameter) for the $k$th
inverter.

\paragraph*{Inverter model} Using~\eqref{eq:inverter_dynamics}, the governing dynamics for \emph{all} inverters
in the network can be expressed as:
\begin{align}
\begin{split}\label{eq:inverters_dynamics2}
\dot{\mathbf{y}} = F(\mathbf{y})+G(\mathbf{y})\mathrm{i}_{\mathrm{oDQ}}+C\mathbf{s}^{\mathrm{ref}},
\end{split}
\end{align}
where, in $\mathbf{y} = (y_1^{\top}, \ldots, y_n^{\top})^{\top}\in \mathbb{R}^{13n}$, $y_k$ captures all the dynamic states for the $k$th inverter,
$F(\mathbf{y})=(f^{\top}_1(y_1) , \ldots,  f^{\top}_n(y_n))^{\top}$,
$G(\mathbf{y})=\diag\left(g_{1}(y_1), \ldots,
  g_n(y_n)\right)$, and $C=\diag\left(C_1,\ldots,C_n\right)$, where
$f_k$ is the drift vector field and $g_k$ and
$C_k$ are control vector fields of inverter $k$. 

\paragraph*{Load and line models} Let $v_k\in \mathbb{R}^2$ be the $k$th load voltage (in the $\mathrm{DQ}$-frame)
and $i_k\in \mathbb{R}^2$ be the current demand (in the $\mathrm{DQ}$-frame) of the $k$th load. We collect the nodal voltages and current demands for the loads in $\subscr{\mathbf{v}}{L} = (v_{1}^{\top},\ldots,
v_{\ell}^{\top})^{\top}\in \mathbb{R}^{2\ell}$ and
$\subscr{\mathbf{i}}{L} = (i_{1}^{\top},\ldots, i_{\ell}^{\top})^{\top}\in \mathbb{R}^{2\ell}$, respectively. We assume loads are purely resistive; suppose $\subscr{\mathbf{R}}{L} \in \mathbb{R}^{\ell}$ is the vector of load resistances, then the loads can be described by
\begin{align}\label{eq:load_dynamics2}
\subscr{\mathbf{v}}{L} = \big(-[\subscr{\mathbf{R}}{L}]\otimes I_2\big)\subscr{\mathbf{i}}{L}.
\end{align}
Suppose that the vector of line resistances and line inductances are denoted by $\mathbf{R}_{\mc E}\in
\mathbb{R}^{m}$ and $\mathbf{L}_{\mc
  E}\in \mathbb{R}^{m}$, respectively; the nodal current injections by $\mathbf{i} =
  (\subscr{\mathrm{i}}{g},
  -\subscr{\mathbf{i}}{L},\subscr{\mathbf{i}}{oDQ})^{\top}$ and nodal
  voltages by $\mathbf{v} = (\subscr{v}{gDQ}, \subscr{\mathbf{\mathbf{v}}}{L},
\subscr{\mathrm{\mathbf{v}}}{oDQ})^{\top}$. The governing dynamics for the transmission
lines are~\cite[Equation 4.10]{JS-DZ-RO-AMS-TS-JR:16}: 
\begin{multline}\label{eq:line_dynamics}
([\mathbf{L}_{\mc E}]\otimes I_2)\subscr{\dot{\xi}}{DQ}  =  \big(-[\mathbf{R}_{\mc E}]\otimes I_2 -
  \subscr{\omega}{nom}[\mathbf{L}_{\mc
    E}]\otimes \mJ\big)\subscr{\xi}{DQ} \\ 
    + (B^{\top}\otimes I_2)\mathbf{v},
\end{multline}
where $\subscr{\xi}{DQ}\in \mathbb{R}^{2m}$ is the vector of
current flows in the lines. Thus we have $\mathbf{i}= (B\otimes I_2)
\subscr{\xi}{DQ}$. 

\subsection{Network Model and Dimensionless Transcription}
From~\eqref{eq:inverters_dynamics2}--\eqref{eq:line_dynamics}, the \emph{grid-tied inverter-network dynamics} are:
\begin{align}\label{eq:network_dynamics}
  \dot{\mathbf{y}}&=F(\mathbf{y})+G(\mathbf{y})(\subscr{B}{I}\otimes I_2)\subscr{\xi}{DQ}+C\mathbf{s}^{\mathrm{ref}},\nonumber\\
  ([\mathbf{L}_{\mc E}]\otimes I_2)\subscr{\dot{\xi}}{DQ}  &=  \big(-[\mathbf{R}_{\mc E}]\otimes I_2 -
  \subscr{\omega}{nom}[\mathbf{L}_{\mc
    E}]\otimes \mJ\big)\subscr{\xi}{DQ} \nonumber\\
    &\phantom{=} + (B^{\top}\otimes I_2)\mathbf{v},
\end{align}
where $\mathbf{v} = (\subscr{v}{gDQ}, \subscr{\mathbf{\mathbf{v}}}{L},
\subscr{\mathrm{\mathbf{v}}}{oDQ})^{\top}$ and $\subscr{\mathbf{v}}{L} =
\big(-\subscr{B}{L}[\subscr{\mathbf{R}}{L}]\otimes I_2\big)\subscr{\xi}{DQ}$.
We transcribe the differential
equations~\eqref{eq:network_dynamics} in a dimensionless
format. We assume that $\subscr{s}{nom}$ is the nominal
  power generation/consumption in the network. For each inverter, we introduce the
following dimensionless variables:
\begin{align*}
\subscr{\widehat{v}}{PLL} &:=
                              \frac{\subscr{v}{PLL}}{\subscr{V}{g}},\quad \subscr{\widehat{\phi}}{PLL} := \frac{\subscr{\mathrm{ki}}{PLL}}{\subscr{V}{g}\subscr{\mathrm{kp}}{PLL}}\subscr{\phi}{PLL},\quad
\subscr{\widehat{s}}{avg} := \frac{\subscr{s}{avg}}{\subscr{s}{nom}}, \\
\subscr{\widehat{\phi}}{s} &:=
  \frac{\subscr{V}{g}\subscr{\mathrm{ki}}{s}\subscr{\phi}{s}}{\subscr{s}{nom}},\quad \widehat{s}^{\mathrm{ref}} := \frac{s^{\mathrm{ref}}}{\subscr{s}{nom}},\quad  
\subscr{\widehat{\gamma}}{DQ}  := \frac{\subscr{\mathrm{ki}}{c}\subscr{\gamma}{DQ}}{\subscr{V}{g}}, \\
\subscr{\widehat{i}}{lDQ}  & :=
                             \frac{\subscr{V}{g}\subscr{i}{lDQ}}{\subscr{s}{nom}}.
\end{align*}
For the network, we introduce the dimensionless parameters
\begin{align*}
  \subscr{\widehat{\xi}}{DQ} := \subscr{V}{g}\subscr{s^{-1}}{nom}\subscr{\xi}{DQ},
                     \quad \widehat{\mathbf{v}} :=
                               \subscr{V^{-1}}{g}\mathbf{v},
                               \quad \widehat{\mathbf{i}} := \subscr{V}{g}\subscr{s^{-1}}{nom}\mathbf{i}.
  \end{align*}
We also isolate time-constants of different sub-systems:
\begin{itemize}[noitemsep,nolistsep,itemindent=0pt,wide=0pt]
\item  for the PLL low-pass filter: $\subscr{\tau}{PLL} =
  \subscr{\omega^{-1}}{c,PLL}$ and $\subscr{\tau'}{PLL} =
  (\subscr{V}{g}\subscr{\mathrm{kp}}{PLL})^{-1}$;
  
\item for the PLL PI controller:
  $\subscr{\mathrm{T}}{PLL} =
  \frac{\subscr{\mathrm{kp}}{PLL}}{\subscr{\mathrm{ki}}{PLL}}$;
  
\item for the low-pass filter 
of the power controller:
  $\subscr{\tau}{s} = \subscr{\omega^{-1}}{s}$;
\item for tracking in the power controller: $\subscr{\tau'}{s} = (\subscr{V}{g}\subscr{\mathrm{ki}}{s})^{-1}$;
\item for the PI controller in the power controller: $\subscr{\mathrm{T}}{s} = \frac{\subscr{\mathrm{kp}}{s}}{\subscr{\mathrm{ki}}{s}}$;
\item for the current controller: $\subscr{\tau}{c} =
  \subscr{V}{g}^2(\subscr{\mathrm{ki}}{c}\subscr{s}{nom})^{-1}$;
\item for the PI controller in the current controller:
  $\subscr{\mathrm{T}}{c} =
  \frac{\subscr{\mathrm{kp}}{c}}{\subscr{\mathrm{ki}}{c}}$;
\item for the LC filter:
    $\subscr{\tau}{LC}
    =\frac{\subscr{\mathrm{L}}{f}}{\sqrt{\subscr{\omega^{-2}}{nom}\mathrm{C}^{-2}_{\mathrm{f}}+
        \subscr{\omega^2}{nom}\subscr{\mathrm{L}^2}{f}}}$ and
    $\subscr{\tau'}{LC}=\frac{\subscr{\mathrm{C}}{f}}{\sqrt{\subscr{\omega^{2}}{nom}\mathrm{C}^{2}_{\mathrm{f}}+
        \subscr{\omega^{-2}}{nom}\subscr{\mathrm{L}^{-2}}{f}}}$ and
    $\subscr{\tau''}{LC} = 
  \subscr{\mathrm{C}}{f}\subscr{\omega}{nom}\subscr{s^{-1}}{nom}\subscr{V}{g}^{2}$;
  \item for line $e\in \mathcal{E}$ in the network: $\subscr{\tau}{e}=\frac{L_e}{\sqrt{R^2_e+\subscr{\omega^2}{nom}L_e^2}}$ and $\subscr{\tau'}{e}=\subscr{s}{nom}\subscr{V}{g}^{-2}\sqrt{R^2_e+\subscr{\omega^2}{nom}L_e^2}$.
\end{itemize}
With these preliminaries in place, the \emph{dimensionless grid-tied inverter-network dynamics} can be expressed as:
\begin{subequations}\label{eq:network_dynamics_dim}
\begin{align}
&\hspace{-0.2cm}\dot{\widehat{\mathbf{v}}}_{\mathrm{PLL}} =[\subscr{\boldsymbol{\tau}}{PLL}]^{-1}(
\widehat{\mathbf{v}}_{\mathrm{od}} - \widehat{\mathbf{v}}_\mathrm{PLL}
), \label{eq1}\\
&\hspace{-0.2cm}\dot{\widehat{\boldsymbol{\phi}}}_\mathrm{PLL} = -[\subscr{\mathbf{T}}{PLL}]^{-1}\widehat{\mathbf{v}}_\mathrm{PLL}, \label{eq2}\\
&\hspace{-0.2cm}\dot{\boldsymbol{\delta}} = 
[\subscr{\boldsymbol{\tau}'}{PLL}]^{-1}\left( 
  \widehat{\boldsymbol{\phi}}_\mathrm{PLL} - \widehat{\mathbf{v}}_\mathrm{PLL} \right), \label{eq3}\\
&\hspace{-0.2cm}\dot{\widehat{\mathbf{s}}}_{\mathrm{avg}} =
[\subscr{\boldsymbol{\tau}}{s}\otimes I_2]^{-1}(\widehat{\mathbf{s}} - \widehat{\mathbf{s}}_{\mathrm{avg}}), \label{eq4}\\ 
  &\hspace{-0.2cm}\dot{\subscr{\widehat{\boldsymbol{\phi}}}{s}}   =
[\subscr{\boldsymbol{\tau}'}{s}\otimes I_2]^{-1}(\widehat{\mathbf{s}}^{\mathrm{ref}} -
                                   \widehat{\mathbf{s}}_{\mathrm{avg}}), \label{eq5}\\
&\hspace{-0.2cm}\subscr{\dot{\widehat{\boldsymbol{\gamma}}}}{DQ} =
[\subscr{\boldsymbol{\tau}}{c}\otimes I_2]^{-1}
                                                 \left(\subscr{\tilde{\mathbf{i}}}{lDQ}-\subscr{\widehat{\mathbf{i}}}{lDQ}\right)
                                                + \mJ[\dot{\boldsymbol{\delta}}]\subscr{\widehat{\boldsymbol{\gamma}}}{DQ} , \label{eq6}\\ 
&\hspace{-0.2cm}\dot{\widehat{\mathbf{i}}}_{\mathrm{lDQ}}
                                                                                                                                                   =
([\subscr{\boldsymbol{\tau}}{LC}]^{-1}[\mathcal{X}]\otimes I_2)\left( \widehat{\mathbf{v}}_{\mathrm{lDQ}}
                        -
                                   \widehat{\mathbf{v}}_{\mathrm{oDQ}}
                                   \right) +
                                                                                                                                                    \mJ [\dot{\boldsymbol{\delta}}] \widehat{\mathbf{i}}_{\mathrm{lDQ}}, \label{eq7}\\ 
&\hspace{-0.2cm}\dot{\widehat{\mathbf{v}}}_{\mathrm{oDQ}} 
=([\subscr{\boldsymbol{\tau}'}{LC}][\mathcal{X}]\otimes I_2)^{-1}\big(
                                   \widehat{\mathbf{i}}_{\mathrm{lDQ}} -
                                   \widehat{\mathbf{i}}_{\mathrm{oDQ}} \nonumber
  \\ &\hspace{4cm} - [\subscr{\boldsymbol{\tau}''}{LC}\otimes I_2]\mJ\widehat{\mathbf{v}}_{\mathrm{oDQ}}\big),\label{eq8}\\
&\hspace{-0.2cm}\subscr{\dot{\widehat{\xi}}}{DQ}  =
                                                                                                                                                                                           \big([\boldsymbol{\tau}_{\mathcal{E}}][\boldsymbol{\tau}'_{\mathcal{E}}]\otimes
                                                                                                                                                                                           I_2\big) ^{-1}\big((B^{\top}\otimes
                                                      I_2)\mathbf{\widehat{v}}
                                                                                                                                                                                      -
                                                                                                                                                                                      \mathcal{Z}\subscr{\widehat{\xi}}{DQ}\big).\label{eq9}                                                                           
\end{align}
\end{subequations}
Above,  
\begin{align*}
\tilde{\mathbf{i}}_{\mathrm{lDQ}} &= [\subscr{\mathbf{T}}{s}\otimes I_2]\mR(-\boldsymbol{\delta}) \mH\subscr{\dot{\widehat{\boldsymbol{\phi}}}}{s}
                                                   +
                                                  \mR(-\boldsymbol{\delta}) \mH\subscr{\widehat{\boldsymbol{\phi}}}{s},\\
\widehat{\mathbf{v}}_{\mathrm{lDQ}} &= [\subscr{\mathbf{T}}{c}\otimes I_2]\subscr{\dot{\widehat{\boldsymbol{\gamma}}}}{DQ}
                                   +
                                      \subscr{\widehat{\boldsymbol{\gamma}}}{DQ} ,\\
  \widehat{\mathbf{i}} & = (B\otimes I_2)\subscr{\widehat{\xi}}{DQ},\\
  \mathcal{Z} & = [\subscr{\omega}{nom}\mathbf{L}_{\mc
                E}]^{-1}[\mathbf{R}_{\mc E}]\otimes I_2 + (I_m\otimes
                \mJ),\\
  \mathcal{X} & = \subscr{s^{-1}}{nom}\subscr{V^2}{g}[\subscr{\omega^{-2}}{nom}\mathbf{C}^{-2}_{\mathrm{f}}+
        \subscr{\omega^2}{nom}\subscr{\mathbf{L}^2}{f}]^{\frac{1}{2}},\\
  \mathbf{\widehat{v}} &= (\subscr{\widehat{v}}{gDQ}, \subscr{\mathbf{\widehat{v}}}{L},
\subscr{\widehat{\mathbf{v}}}{oDQ})^{\top}.
\end{align*}
The dimensionless grid-tied inverter network dynamics~\eqref{eq1}--\eqref{eq9} is $(13n+2m)$-dimensional. Our first goal is to find the equilibrium points of the dynamical systems~\eqref{eq1}--\eqref{eq9}.
\subsection{Equilibrium points of the Dimensionless Grid-tied
  Inverter-network Dynamics}

 We start by introducing some notation. Let $Y\in \mathbb{R}^{(2n+2)}$ be the admittance
matrix of the network. Then 
\begin{align} \label{eq:defins}
\begin{split}
\subscr{Y}{red} &:= \subscr{Y}{II} -
\subscr{Y}{IL}\left(\subscr{Y}{LL}+[\subscr{\mathbf{R}}{L}]^{-1}\otimes
I_2\right)^{-1}
\subscr{Y}{LI}, \\
\subscr{Y}{C,red} &:= \subscr{Y}{red} +
\subscr{\omega}{nom}\left([\subscr{\mathbf{C}}{f}]\otimes J\right),\\ 
\subscr{Y}{g} &:= \subscr{Y}{I0} - \subscr{Y}{IL}\left(\subscr{Y}{LL}
  +[\subscr{\mathbf{R}}{L}]^{-1}\otimes I_2 \right)^{-1}\subscr{Y}{L0}
, \\
\mathbf{w} &:= -\subscr{Y^{-1}}{red}\subscr{Y}{0g}\subscr{v}{gDQ},
\end{split}
\end{align}
and the dimensionless parameters:
\begin{align*}
  \widehat{Y}_{(\cdot)} &:= \subscr{V}{g}^2\subscr{\mathrm{s}^{-1}}{nom} {Y}_{(\cdot)},\qquad
  \widehat{\mathbf{w}} :=
  -\subscr{V}{g}\subscr{\mathrm{s}^{-1}}{nom}\mathbf{w},\\
  \mathcal{Z}_{\mathrm{L}} &= \mathcal{Z} +
\big([\subscr{\omega}{nom}\mathbf{L}_{\mathcal{E}}]^{-1}\subscr{B^{\top}}{L}[\subscr{\mathbf{R}}{L}]\subscr{B}{L}\big)\otimes I_2.
\end{align*}
We show that the equilibrium points of the dimensionless grid-tied
inverter network dynamics~\eqref{eq:network_dynamics_dim} are in
correspondence with the solutions of the following power-flow equations:
\begin{align}
\widehat{\mathbf{s}}^{\mathrm{ref}}& = \tfrac{3}{2}\mD(\subscr{\widehat{\mathbf{v}}}{oDQ})\subscr{\widehat{\mathbf{i}}}{oDQ}, \label{eq:power_flow1}\\
\subscr{\widehat{\mathbf{i}}}{oDQ} & = \subscr{\widehat{Y}}{red}\subscr{\widehat{\mathbf{v}}}{oDQ} + \subscr{\widehat{Y}}{g}\subscr{\widehat{v}}{gDQ}. \label{eq:power_flow2}
\end{align}
\begin{lemma}\label{lem:power_flow_equilibrium_correspondence}
For a given reference-power injection $\widehat{s}^{\mathrm{ref}}$ to the
inverters, the following statements are equivalent:
\begin{enumerate}
\item\label{p1:power_flow}
  $(\subscr{\widehat{\mathbf{v}}^{\mathrm{ref}}}{oDQ},\subscr{\widehat{\mathbf{i}}^{\mathrm{ref}}}{oDQ})^{\top}\in
  \mathbb{R}^{4n}$ is a solution for the power-flow equations~\eqref{eq:power_flow1} and~
  \eqref{eq:power_flow2};
\item\label{p2:equilibrium_point} for every $\alpha=(\alpha_1,\ldots,\alpha_n)^{\top}\in
  \mathbb{Z}^n_2$, $\widehat{\mathbf{x}}^{\mathrm{ref}}_{\alpha}\in \mathbb{R}^{13n+2m}$ is an
  equilibrium point of the dimensionless grid-tied inverter-network dynamics~\eqref{eq:network_dynamics_dim}
  given by 
\begin{align*}
&\hspace{-1cm} \widehat{\mathbf{x}}^{\mathrm{ref}}_{\alpha} = (
\vect{0}_{2n},
\boldsymbol{\delta}^{\mathrm{ref}}+\alpha\pi,
(-1)^{\alpha}\subscr{\widehat{\boldsymbol{\phi}}^{\mathrm{ref}}}{s},
\widehat{\mathbf{s}}^{\mathrm{ref}},
\widehat{\boldsymbol{\gamma}}^{\mathrm{ref}}_{\mathrm{DQ}},
\widehat{\mathbf{i}}^{\mathrm{ref}}_{\mathrm{lDQ}},
  \widehat{\mathbf{v}}^{\mathrm{ref}}_{\mathrm{oDQ}}, 
  \subscr{\widehat{\xi}^{\mathrm{ref}}}{DQ})^{\top}
\end{align*}
with 
\begin{align*}
&\boldsymbol{\delta}^{\mathrm{ref}}=-\mathrm{arg}(\widehat{\mathbf{v}}^{\mathrm{ref}}_{\mathrm{oDQ}}), 
\hspace{0.5cm}\widehat{\mathbf{i}}^{\mathrm{ref}}_{\mathrm{lDQ}}=[\subscr{\boldsymbol{\tau}''}{LC}\otimes
                                                                                                                        I_2]\mJ\widehat{\mathbf{v}}^{\mathrm{ref}}_{\mathrm{oDQ}}+\widehat{\mathbf{i}}^{\mathrm{ref}}_{\mathrm{oDQ}},\\
&\subscr{\widehat{\boldsymbol{\phi}}^{\mathrm{ref}}}{s}=\mH\mR(\boldsymbol{\delta}^{\mathrm{ref}})\widehat{\mathbf{i}}^{\mathrm{ref}}_{\mathrm{lDQ}}
                                             ,\hspace{0.3cm}\widehat{\boldsymbol{\gamma}}^{\mathrm{ref}}_{\mathrm{DQ}}=\widehat{\mathbf{v}}^{\mathrm{ref}}_{\mathrm{oDQ}},\\
&\subscr{\widehat{\xi}^{\mathrm{ref}}}{DQ} =
                                                                                                           \subscr{\mathcal{Z}^{-1}}{L}\big((\subscr{B^{\top}}{I}\otimes I_2)
\widehat{\mathbf{v}}^{\mathrm{ref}}_{\mathrm{oDQ}} + (\subscr{B^{\top}}{0}\otimes I_2)
\widehat{v}_{\mathrm{gDQ}}\big).
\end{align*}
\end{enumerate}
\end{lemma}
\begin{proof}
Regarding $(ii) \Longrightarrow (i)$, from the power-controllers' dynamics in~\eqref{eq:network_dynamics_dim}, we can conclude that if $\widehat{\mathbf{x}}^{\mathrm{ref}}_{\alpha}$ is an equilibrium point,
then we have $\widehat{\mathbf{s}}=\widehat{\mathbf{s}}^{\mathrm{ref}}$. This implies
$(\subscr{\widehat{\mathbf{v}}^{\mathrm{ref}}}{oDQ}, \subscr{\widehat{\mathbf{i}}^{\mathrm{ref}}}{oDQ})^{\top}$ 
satisfies~\eqref{eq:power_flow1}. Moreover, at the equilibrium point
$\widehat{\mathbf{x}}^{\mathrm{ref}}_{\alpha}$, the line dynamics~\eqref{eq9} will
simplify to $\widehat{\mathbf{i}}^{\mathrm{ref}} = \widehat{Y} \widehat{\mathbf{v}}^{\mathrm{ref}}$. Using Kron
reduction~\cite{FD-FB:11d}, this implies that $(\subscr{\widehat{\mathbf{v}}^{\mathrm{ref}}}{oDQ},
\subscr{\widehat{\mathbf{i}}^{\mathrm{ref}}}{oDQ})^{\top}$ satisfies~\eqref{eq:power_flow2}. 

Regarding $(i) \Longrightarrow (ii)$, suppose
 $(\subscr{\widehat{\mathbf{v}}^{\mathrm{ref}}}{oDQ},\subscr{\widehat{\mathbf{i}}^{\mathrm{ref}}}{oDQ})^{\top}$ is a
solution to the power flow
equations~\eqref{eq:power_flow1} and~\eqref{eq:power_flow2}. Then from
the PLL dynamics in~\eqref{eq:network_dynamics_dim}, we have $\widehat{\mathbf{v}}_{\mathrm{PLL}}=\vect{0}_n$ and
$\widehat{\mathbf{v}}_{\mathrm{od}}=\vect{0}_n$. Note that, $\widehat{\mathbf{v}}_{\mathrm{od}}=\vect{0}_n$ can be written in the trigonometric form 
\begin{equation*}
v^{k,\mathrm{ref}}_{\mathrm{oD}}\cos(\delta^{k,{\mathrm{ref}}})+v^{k,{\mathrm{ref}}}_{\mathrm{oQ}}\sin(\delta^{k, \mathrm{ref}})=0.
\end{equation*}
This implies that, for every $k\in\{1,\ldots,n\}$, there exists $\alpha_k\in \mathbb{Z}_2$ such that 
$\delta^{k,{\mathrm{ref}}}=-\mathrm{arg}(v^{k,{\mathrm{ref}}}_{\mathrm{oDQ}})+\alpha_k\pi$. From the power-controller dynamics in~\eqref{eq:network_dynamics_dim}, we have $\subscr{\widehat{\mathbf{s}}}{avg}=\widehat{\mathbf{s}}^{\mathrm{ref}}$.
Finally, one can find the value of $\subscr{\widehat{\boldsymbol{\phi}}^{\mathrm{ref}}}{s}$,
$\subscr{\widehat{\mathbf{i}}^{\mathrm{ref}}}{lDQ}$,
$\subscr{\widehat{\boldsymbol{\gamma}}^{\mathrm{ref}}}{DQ}$, and $\subscr{\widehat{\xi}^{\mathrm{ref}}}{DQ}$ by solving the remaining equations in~\eqref{eq:network_dynamics_dim}. 
\end{proof}

\begin{remark}
The power-flow equations~\eqref{eq:power_flow1} and~\eqref{eq:power_flow2} have been studied
extensively in the literature and many sufficient conditions for
existence and uniqueness of solutions have been developed; see,
e.g.,~\cite{SB-SZ:16,SVD-SSG-YCC:15,CW-AB-JYLB-MP:16}. Lemma~\ref{lem:existence_algorithm}
in Appendix~\ref{app:powerflow} restates a result
from~\cite{CW-AB-JYLB-MP:16} pertaining to uniqueness that is
leveraged in subsequent results.
\end{remark}

\section{Stability Analysis of the Dimensionless Grid-tied Inverter-network Dynamics}\label{sec:stability_analysis}

In bulk power-systems dynamics literature,
it is commonplace to assume that the dynamics of the grid-following
inverters are much faster than the dynamics of grid-forming inverters and
synchronous machines~\cite{JS-DZ-RO-AMS-TS-JR:16}.
This assumption justifies the use of a static model for grid-following inverters such that the inverter nodes are considered to be sources of constant (active and reactive) power. Subsequently, the network operation is described by the following power-flow equations:
\begin{align}\label{eq:static_model_inverter}
\vect{0}_{2} = \widehat{s}^{\mathrm{ref}} - \tfrac{3}{2}\mD(\subscr{\widehat{v}}{oDQ}) \subscr{\widehat{i}}{oDQ}.
\end{align}
Quite obviously, the static
representation~\eqref{eq:static_model_inverter} does not capture stability. The following example shows that the internal dynamics of the inverters can induce instabilities, even if the power-flow equations in~\eqref{eq:static_model_inverter} admit a high-voltage solution.

\begin{example}\emph{(\textbf{Instabilities Induced by Inverter Dynamics})}\label{exm:dynamic_vs_static}
Consider the radial grid-connected network consisting of $25$ identical
inverters (a sketch is provided in Fig.~\ref{fig:inverter}). Suppose
the inverters have uniform reference power injections $\mathbf{p}^{\mathrm{ref}} =
p\vect{1}_{25}$, each line has resistance $R=10^{-2}\ \Omega$ and inductance $L=10^{-5}\ \mathrm{H}$,
and the grid voltage is $\subscr{v}{gDQ} = [0, 120\sqrt{2}]^{\top}\ \mathrm{V(peak)}$ with constant frequency $\omega_{\mathrm{nom}} = 120\pi\ \mathrm{rad/s}$. 

\begin{enumerate}
  \item \textbf{Static model:} If inverter power injections satisfy
    \begin{align}\label{eq:power_flow_stability}
      \widehat{p}\left\|\subscr{\widehat{Y}^{-1}}{red}\right\|_{\mathbb{C},\infty}\le \tfrac{3}{8},
    \end{align}
then, there exists a unique high-voltage, low-current solution for the power-flow equations~\eqref{eq:static_model_inverter} (see Lemma~\ref{lem:existence_algorithm}).

  \item \textbf{Dynamic model:} We use the dynamic model~\eqref{eq:inverter_dynamics} for the inverters
(parameters are given in the fourth
  column of Table~\ref{tab:parameters}). The governing
  equations for the network are in~\eqref{eq:network_dynamics_dim},
  and condition~\eqref{eq:power_flow_stability} via Lemma~\ref{lem:power_flow_equilibrium_correspondence} 
  guarantees the existence of a family of equilibrium points $\widehat{\mathbf{x}}^{\mathrm{ref}}_{\alpha}$, $\alpha\in \mathbb{Z}^n_2$ for the system~\eqref{eq:network_dynamics_dim}. Linearizing the system~\eqref{eq:network_dynamics_dim}, we study 
  local stability of the equilibrium point  $\widehat{\mathbf{x}}^{\mathrm{ref}}_{\mathrm{0}}$.
\end{enumerate}
  Figure~\ref{example} plots the maximum real part of the eigenvalues of the linearized
  system~\eqref{eq:network_dynamics_dim} around the equilibrium point
  $\widehat{\mathbf{x}}^{\mathrm{ref}}_{\mathrm{0}}$ as a function of active-power injection. The red vertical line is the threshold
  of the power injection for which the power-flow equations admit a unique solution (obtained from~\eqref{eq:power_flow_stability}). Notice that there are power injections for which a high-voltage solution of the power-flow equations exists, however, the corresponding equilibrium point $\widehat{\mathbf{x}}^{\mathrm{ref}}_{\mathrm{0}}$ is not stable.
\end{example}

\subsection{Small-signal Stability via Time-scale Separation}
We now focus on the small-signal stability of the dimensionless grid-tied inverter-network dynamics~\eqref{eq:network_dynamics_dim}. Due to high dimensionality and nonlinearity of the dynamic model, studying small-signal stability is not analytically tractable. Therefore, it is a reasonable goal to reduce the model order. To that end, we identify a 
physically meaningful parametrization of the inverters, and we show that under suitable assumptions, this parametrization leads
to a time-scale decomposition of the
system~\eqref{eq:network_dynamics_dim} which simplifies analysis. We begin by uncovering time constants of different sub-systems. 

\begin{definition}[Singular perturbation parameter]\label{assu2}
For dynamics~\eqref{eq:network_dynamics_dim}, we define:
\begin{align*}
&\epsilon_{\mathrm{I}} =  \max\{\|\subscr{\boldsymbol{\tau}}{PLL}\|_{\infty}, \|\subscr{\boldsymbol{\tau}'}{PLL}\|_{\infty},
  \|\subscr{\boldsymbol{\tau}}{s}\|_{\infty},\|\subscr{\mathbf{T}}{s}\|_{\infty},
   \|\subscr{\mathbf{T}}{PLL}\|_{\infty},\\
  &\|\subscr{\boldsymbol{\tau}}{c}\|^{\frac{1}{2}}_{\infty},
  \|\subscr{\mathbf{T}}{c}\|^{\frac{1}{2}}_{\infty}\},\\
  &\epsilon_{\mathcal{E}}=\max\left\{\|\subscr{\boldsymbol{\tau}}{LC}\|^{\frac{1}{2}}_{\infty}, \|\subscr{\boldsymbol{\tau}'}{LC}\|^{\frac{1}{2}}_{\infty},\|\boldsymbol{\tau}_{\mathcal{E}}\|^{\frac{1}{2}}_{\infty}\right\}.
\end{align*}
Using these, we define the \emph{singular perturbation parameter}:
  \begin{equation} \label{eq:eps}
\epsilon := \max\{\epsilon_{\mathrm{I}}, \epsilon_{\mathcal{E}}\}.
\end{equation}
\end{definition}
\begin{remark} \label{rem:good}
\begin{enumerate}
    \item Definition~\ref{assu2} establishes a
  physically meaningful time-scale separation of the components of the
  inverter network when $\epsilon\ll 1$. In this case, 
  $ \|\boldsymbol{\tau}'_{\mathrm{LC}}\|_{\infty},\|\boldsymbol{\tau}_{\mathrm{LC}}\|_{\infty}, \|\boldsymbol{\tau}_{\mathcal{E}}\|_{\infty},
  \|\subscr{\boldsymbol{\tau}}{c}\|_{\infty},
  \|\subscr{\mathbf{T}}{c}\|_{\infty} \le \epsilon^2$ which
  implies that the line dynamics, the $LC$ filter, and the
  current controller of the inverter are the fastest
  components of the network. Moreover, $\|\subscr{\boldsymbol{\tau}}{PLL}\|_{\infty},
  \|\subscr{\boldsymbol{\tau}'}{PLL}\|_{\infty},\|\subscr{\boldsymbol{\tau}}{s}\|_{\infty}\le
  \epsilon$ which implies that the
  PLL, and averaging-part of the power controller (i.e.,
  $\widehat{\mathbf{s}}_\mathrm{avg}$) are slower than the current
  controller, the line dynamics, and the $LC$ filter but they are faster
  than the steady-state power-tracking controller (i.e., $\widehat{\boldsymbol{\phi}}_\mathrm{s}$) dynamics.
\item  The assumption $\epsilon\ll 1$ is equivalent to
  $\subscr{\epsilon}{I}, \subscr{\epsilon}{L}\ll
  1$, which is realistic in practice. For instance, in the 
  inverter models with the parameters in
Table~\ref{tab:parameters} and with the
parameters used in~\cite{NP-MP-TCG:07,MR-JAM-JWK:14,MR-JAM-JWK:15}, it holds that
$\subscr{\epsilon}{I},\subscr{\epsilon}{L}\le 0.1$. 
\item\label{up} $\tau_{e} =
  \frac{\subscr{L}{e}}{\sqrt{\subscr{R^2}{e} +
      \subscr{\omega^2}{nom}\subscr{L^2}{e}}} \le
  \subscr{\omega^{-1}}{nom}$, for every $e\in
  \mathcal{E}$. Similarly, one can show that
  $\|\subscr{\boldsymbol{\tau}'}{LC}\|_{\infty}, \|\subscr{\boldsymbol{\tau}}{LC}\|_{\infty}\le
  \subscr{\omega^{-\frac{1}{2}}}{nom}$ and this gives an upper bound
  for the LC and line parameter $\epsilon_{\mathcal{E}} \le \subscr{\omega^{-\frac{1}{2}}}{nom}$.

\item For a different set of parameters and control architectures, one can conceivably identify a different singular perturbation parameter and time-scale decomposition. However, we expect the general nature of the stability result that follows to be similar.
\end{enumerate}
\end{remark}

\begin{figure}[t!]
\centering
\vspace{-0.2in}
\includegraphics[width=0.8\textwidth]{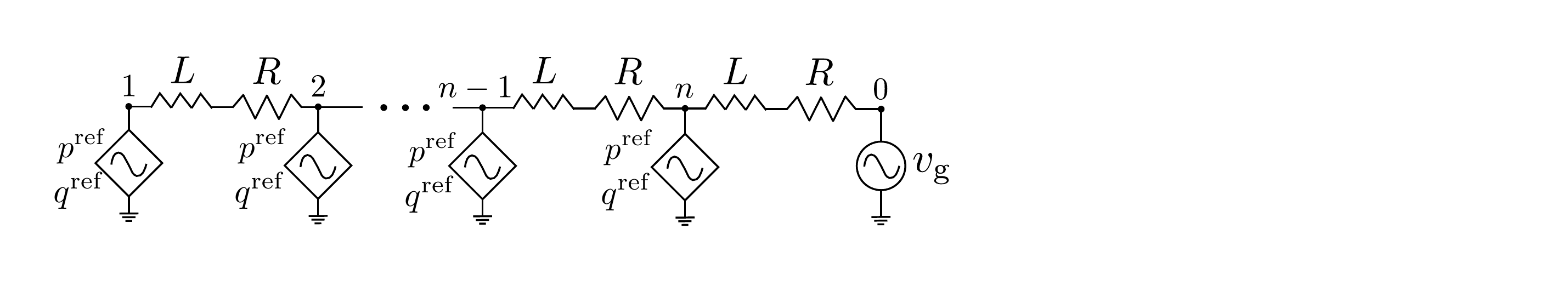} \vspace{-0.4in}
\caption{A radial network with $n$ inverters connected to the grid bus.}
\label{fig:inverter}
\end{figure}
\begin{figure}[t!]\centering
\includegraphics[width=0.9\linewidth]{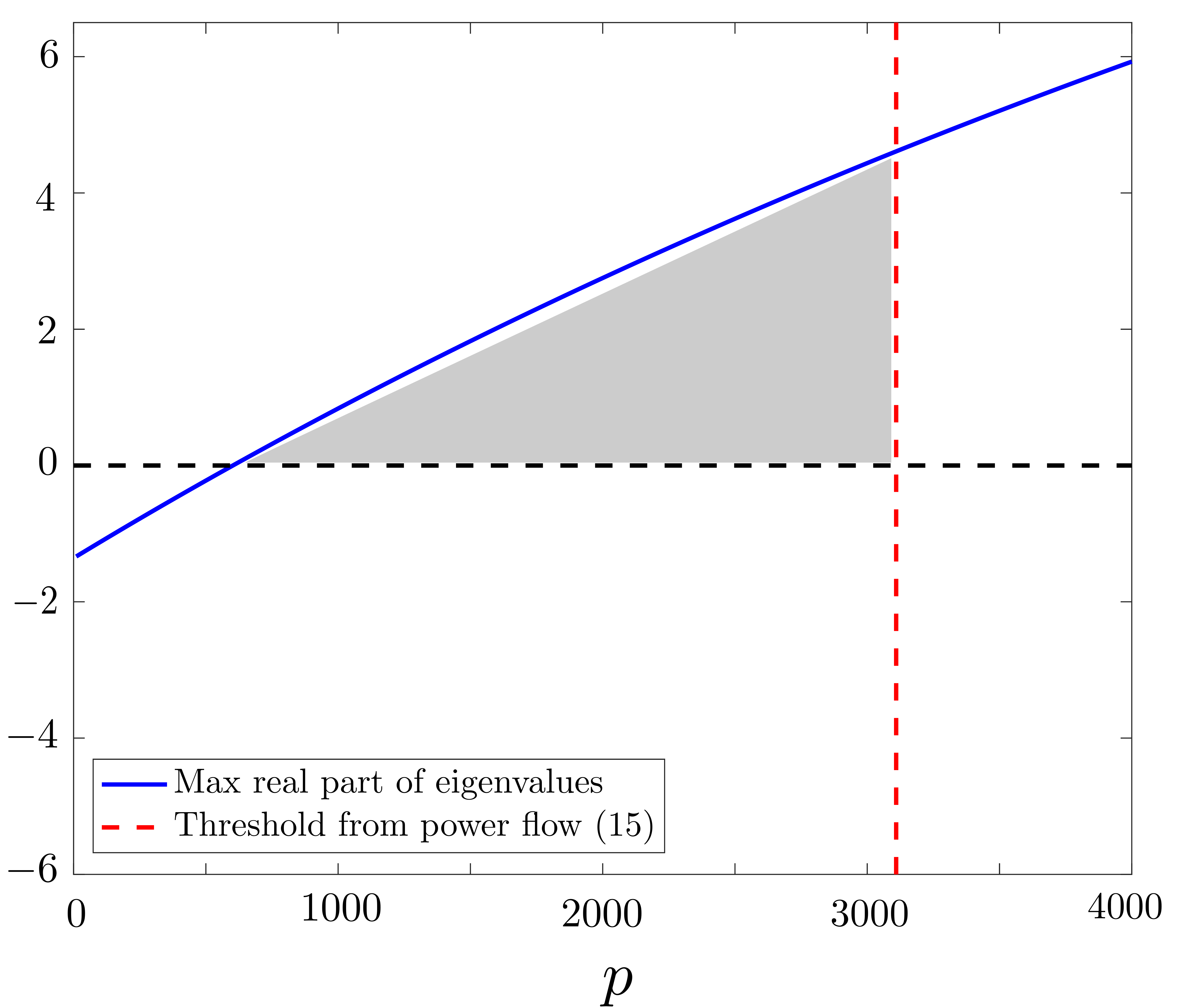}
\caption{Shaded region shows power injections that admit a unique power flow solution that is, however, not small-signal stable.}
\label{example}
\end{figure}

\begin{theorem}[Small-signal Stability]\label{thm:existence_stability}
Consider the dimensionless grid-tied inverter-network
dynamics~\eqref{eq:network_dynamics_dim} with states in $\mathbb{R}^{13n+2m}$ and
references $\widehat{\mathbf{s}}^{\mathrm{ref}}\in \mathbb{R}^{2n}$. The following hold:
\begin{enumerate}
\item\label{p1:exitence} if,
\begin{align}\label{c1:existence}
\left\|\mD'(\widehat{\mathbf{w}})\subscr{\widehat{Y}^{-1}}{red}(\mD'(\widehat{\mathbf{w}}))^{-1}\mD'(\widehat{\mathbf{s}}^{\mathrm{ref}})\right\|_{\mathbb{C},\infty}\le
  \tfrac{3}{8},
\end{align}
then, there is a unique solution
$(\subscr{\widehat{\mathbf{v}}^{\mathrm{ref}}}{oDQ},\subscr{\widehat{\mathbf{i}}^{\mathrm{ref}}}{oDQ})^{\top}$ for
the power-flow equations~\eqref{eq:power_flow1} and~\eqref{eq:power_flow2} and a family of equilibrium points $\widehat{\mathbf{x}}^{\mathrm{ref}}_\alpha$,
$\alpha\in \mathbb{Z}_{2}^{n}$, for the
grid-tied inverter network  dynamics~\eqref{eq:network_dynamics_dim} satisfying:
\begin{align*}
\left\|\subscr{\widehat{\mathbf{v}}^{\mathrm{ref}}}{oDQ}-\widehat{\mathbf{w}} \right\|_{\mathbb{C},\infty}\le \tfrac{1}{2}\|\widehat{\mathbf{w}}\|_{\mathbb{C},\infty};
\end{align*} 
\item\label{p2:stability} if additionally,   $[\subscr{\mathbf{T}}{PLL}]\succ [\subscr{\boldsymbol{\tau}}{PLL}]$
  and the $2n\times 2n$ matrix
  \begin{multline}\label{c2:stability}
\hspace{-0.3cm} M :=
-\Big(\mD(\subscr{\widehat{\mathbf{v}}^{\mathrm{ref}}}{oDQ})
\subscr{\widehat{Y}}{red}+\mD'(\subscr{\widehat{\mathbf{i}}^{\mathrm{ref}}}{oDQ})\Big)
\times \\ \subscr{\widehat{Y}}{C,red}^{-1}\mathbf{R}(-\boldsymbol{\delta}^{\mathrm{ref}}) \mH[\subscr{\boldsymbol{\tau}'}{s}\otimes
    I_2]^{-1}
  \end{multline}
  is Hurwitz, then, there exists an $\epsilon^*>0$ such that, for every $\epsilon\le
 \epsilon^*$, the equilibrium point $\widehat{\mathbf{x}}_{\mathrm{0}}^{\mathrm{ref}}$
 is locally exponentially stable. 
\end{enumerate}
\end{theorem}
\begin{proof}
Regarding part~\ref{p1:exitence}, the proof follows from combining Lemma~\ref{lem:existence_algorithm} and
  Lemma~\ref{lem:power_flow_equilibrium_correspondence}. Regarding
  part~\ref{p2:stability}, consider the dimensionless grid-tied
  inverter network
  dynamics~\eqref{eq:network_dynamics_dim}. Using~\eqref{eq:eps} and defining the
    variable $\Delta \mathbf{x}\in \mathbb{R}^{13n+2m}$ by $\Delta \mathbf{x}
    = \mathbf{x} - \mathbf{x}^{\mathrm{ref}}_{0}$, we get a three-time-scale
  decomposition of the system. We define the states $\mathbf{z}_1\in \mathbb{R}^{6n+2m}$,
  $\mathbf{z}_2\in \mathbb{R}^{5n}$, and $\mathbf{z}_3\in \mathbb{R}^{2n}$ as
  follows: 
\begin{align*} 
\mathbf{z}_1 & = \begin{pmatrix}\Delta\subscr{\widehat{\boldsymbol{\gamma}}}{DQ}&
  \Delta{\widehat{\mathbf{i}}}_{\mathrm{lDQ}}&\Delta{\widehat{\mathbf{v}}}_{\mathrm{oDQ}}
  & \Delta\subscr{\xi}{DQ}\end{pmatrix}^{\top},\\
 \mathbf{z}_2 &= \begin{pmatrix}{\Delta\widehat{\mathbf{v}}}_{\mathrm{PLL}} & \Delta\widehat{\boldsymbol{\phi}}_{\mathrm{PLL}} &
 \Delta\boldsymbol{\delta} & \Delta{\widehat{\mathbf{s}}}_{\mathrm{avg}}\end{pmatrix}^{\top},\\
\mathbf{z}_3 &= \Delta{\subscr{\widehat{\boldsymbol{\phi}}}{s}},
\end{align*}
where $\mathbf{z}_1$ is \emph{faster} than $\mathbf{z}_2$ and $\mathbf{z}_2$
are \emph{faster} than $\mathbf{z}_3$. The corresponding time-scales are given
by $\tau = t/\epsilon^2$ and $\tau'
=t/\epsilon$. Using these time scales, 
the dimensionless grid-tied inverter network
dynamics~\eqref{eq:network_dynamics_dim} can be written as follows:
   \begin{align}
   \begin{split}
\dot{\mathbf{z}}_3 &=
  g_3(\mathbf{z}_1,\mathbf{z}_2, \mathbf{z}_3,\epsilon),\\
\epsilon\dot{\mathbf{z}}_2 &=
  g_2(\mathbf{z}_1,\mathbf{z}_2, \mathbf{z}_3,\epsilon),\\
\epsilon^2\dot{\mathbf{z}}_1 &=
  g_1(\mathbf{z}_1,\mathbf{z}_2, \mathbf{z}_3,\epsilon), 
  \end{split}
  \end{align} 
where $g_1,g_2,g_3$ are suitably defined functions. The quasi-steady-state manifold for the
time-scale $\tau$ is the manifold $\mathbf{z}_1 = h_1(\mathbf{z}_2, \mathbf{z}_3)$ obtained by solving the algebraic equations 
$g_3(\mathbf{z}_1,\mathbf{z}_2,\mathbf{z}_3,0) = \vect{0}_{(6n+2m)}$~\cite[\S 11.2]{HKK:02}. After some algebraic computations, the quasi-steady-state manifold for the time-scale
$\tau$ is time-invariant and is given by:
\begin{align*}
&\Delta\subscr{\widehat{\boldsymbol{\gamma}}}{DQ}  =
                   \subscr{\widehat{Y}^{-1}}{C,red} \mR(-\boldsymbol{\delta}) \mH\Delta\subscr{\widehat{\boldsymbol{\phi}}}{s}, \\
&\Delta\subscr{\widehat{\mathbf{i}}}{lDQ}  =
                  \mR(-\boldsymbol{\delta}) \mH\Delta\subscr{\widehat{\boldsymbol{\phi}}}{s},\\
&\Delta\subscr{\widehat{\mathbf{v}}}{oDQ} 
                                           =\subscr{\widehat{Y}^{-1}}{C,red} \mR(-\boldsymbol{\delta})\mH\Delta\subscr{\widehat{\boldsymbol{\phi}}}{s},\\
&\Delta \subscr{\widehat{\xi}}{DQ} = \subscr{\mathcal{Z}^{-1}}{L}(\subscr{B^{\top}}{I}\otimes I_2) \subscr{\widehat{Y}^{-1}}{C,red} \mR(-\boldsymbol{\delta})\mH\Delta\subscr{\widehat{\boldsymbol{\phi}}}{s}.
\end{align*}
Since the quasi-steady state manifold for the time-scale $\tau$ is an isolated manifold, the singular perturbation problem is in the standard form. The boundary-layer dynamics are:
\begin{subequations}\label{eq:fastest_bl}
\begin{align}
\hspace{-0.25cm}\frac{d \Delta\subscr{\widehat{\boldsymbol{\gamma}}}{DQ}}{d\tau}&=
-\epsilon^2[\subscr{\boldsymbol{\tau}}{c}\otimes I_2]^{-1} 
                                     \Delta\subscr{\widehat{\mathbf{i}}}{lDQ},\label{eq:fastest_bl_1}\\
\hspace{-0.25cm}\frac{d \Delta\widehat{\mathbf{i}}_{\mathrm{lDQ}}}{d\tau} &=
\epsilon^2([\subscr{\boldsymbol{\tau}}{LC}]^{-1}[\mathcal{X}]\otimes I_2)
\Big(\subscr{\Delta\widehat{\boldsymbol{\gamma}}}{DQ}
-\Delta\widehat{\mathbf{v}}_{\mathrm{oDQ}}\nonumber\\ & -
                                   [\subscr{\boldsymbol{\tau}}{c}\otimes I_2]^{-1}[\subscr{\mathbf{T}}{c}\otimes I_2]{\Delta\widehat{\mathbf{i}}}_{\mathrm{lDQ}}\Big), \label{eq:fastest_bl_2}\\ 
\hspace{-0.25cm}\frac{d \Delta\widehat{\mathbf{v}}_{\mathrm{oDQ}}}{d\tau} 
&=\epsilon^2([\subscr{\boldsymbol{\tau}'}{LC}][\mathcal{X}]\otimes I_2)^{-1}\Big( \Delta\widehat{\mathbf{i}}_{\mathrm{lDQ}} -
                                                            \Delta\widehat{\mathbf{i}}_{\mathrm{oDQ}}\nonumber\\
  & -[\subscr{\boldsymbol{\tau}''}{LC}\otimes I_2]\mJ\widehat{\mathbf{v}}_{\mathrm{oDQ}}\Big),\label{eq:fastest_bl_3}\\
\hspace{-0.25cm}\frac{d \Delta\widehat{\xi}_{\mathrm{DQ}}}{d\tau} &= \epsilon^2
                                                    ([\boldsymbol{\tau}_{\mathcal{E}}][\boldsymbol{\tau}'_{\mathcal{E}}]\otimes
                                                                                                                                                                                                                                     I_2)^{-1}\Big((\subscr{B^{\top}}{I}\otimes
                                                    I_2) \Delta
                                                    \widehat{\mathbf{v}}_{\mathrm{oDQ}}
                                                                                                                                                                                                                                    \nonumber \\
  & 
  - \mathcal{Z}_{\mathrm{L}}\widehat{\xi}_{\mathrm{DQ}}\Big)\label{eq:fastest_bl_4}.
\end{align}
\end{subequations}
We first show that for the boundary-layer dynamics~\eqref{eq:fastest_bl}, the
origin is the exponentially stable equilibrium point. It is easy to see
that since the graph is connected, we have
$\mathrm{Ker}(\subscr{B}{I}^{\top}) = \{\vect{0}_{n}\}$, the matrix $([\subscr{\boldsymbol{\tau}''}{LC}\otimes
I_2]) \mJ$ is skew-symmetric and the matrix $\mathcal{Z}_{\mathrm{L}}+\mathcal{Z}^{\top}_{\mathrm{L}}$ is negative definite. Hence, using Lemma~\ref{lem:hurwitz}, the origin is the locally exponentially stable
point of the boundary-layer dynamics~\eqref{eq:fastest_bl}. 
Similarly, the quasi-steady-state manifold for time-scale $\tau'$ is
$\mathbf{z}_2 = h_2(\mathbf{z}_3)$ obtained by solving the algebraic
equations $g_2(h_1(\mathbf{z}_2,\mathbf{z}_3),\mathbf{z}_2,
\mathbf{z}_3, 0) = \vect{0}_{5n}$. Thus, the quasi-steady-state
manifold for time-scale $\tau'$ is time-invariant and is given by:  
\begin{align*}
&\hspace{-0.3cm}\Delta\widehat{\mathbf{v}}_\mathrm{PLL}  =\vect{0}_n, \\
&\hspace{-0.3cm}\Delta\widehat{\boldsymbol{\phi}}_\mathrm{PLL}  = \vect{0}_n,\\
&\hspace{-0.3cm}\Delta\boldsymbol{\delta} =\boldsymbol{\delta}^{BL}
                            - \mathrm{arg}(\widehat{\mathbf{v}}^{\mathrm{ref}}_{\mathrm{oDQ}}),\\
&\hspace{-0.3cm}\Delta\mathrm{\widehat{\mathbf{s}}}_{\mathrm{avg}} = 
                                                     -\tfrac{3}{2}\mD\big(\subscr{\widehat{Y}^{-1}}{C,red}\mR(-\boldsymbol{\delta}^{BL}) \mH\Delta\subscr{\widehat{\boldsymbol{\phi}}}{s}
                                                                                                    +
                                                                                                    \widehat{\mathbf{v}}^{\mathrm{ref}}_{\mathrm{oDQ}}\big)
                                                                                                    \times
                                                                                                    \\
  & \left(\subscr{\widehat{Y}}{red}\subscr{\widehat{Y}^{-1}}{C,red}\mR(-\boldsymbol{\delta}^{BL})\mH\Delta\subscr{\widehat{\boldsymbol{\phi}}}{s}
                                                                                                    +
                                                                                                    \widehat{\mathbf{i}}^{\mathrm{ref}}_{\mathrm{oDQ}}\right)
\end{align*}
where
$\boldsymbol{\delta}^{BL}=\boldsymbol{\delta}^{BL}(\Delta\subscr{\widehat{\boldsymbol{\phi}}}{s})$
is the solution to:
\begin{align*}
\boldsymbol{\delta} -
  \mathrm{arg}\left(\subscr{\widehat{Y}^{-1}}{C,red}\mR(-\boldsymbol{\delta})\mH\Delta\subscr{\widehat{\boldsymbol{\phi}}}{s}\right)
  = \vect{0}_{n}.
  \end{align*}
Since the quasi-steady state manifold for the time-scale $\tau'$ is an isolated manifold, the
singular perturbation problem is in the standard form with boundary-layer dynamics: 
\begin{subequations}\label{eq:medium_bl}
\begin{align}
\frac{d \Delta\widehat{\mathbf{v}}_{\mathrm{PLL}}}{d\tau'} &=
\epsilon[\subscr{\boldsymbol{\tau}}{PLL}]^{-1} (
\Delta\widehat{\mathbf{v}}_{\mathrm{od}} - \Delta\widehat{\mathbf{v}}_\mathrm{PLL}
), \\
\frac{d \Delta\widehat{\boldsymbol{\phi}}_{\mathrm{PLL}}}{d\tau'} &= -\epsilon[\subscr{\mathbf{T}}{PLL}]^{-1}\Delta\widehat{\mathbf{v}}_\mathrm{PLL}, \\
\frac{d \Delta\boldsymbol{\delta}}{d\tau'} &=
\epsilon[\subscr{\boldsymbol{\tau}'}{PLL}]^{-1}(\Delta \widehat{\boldsymbol{\phi}}_\mathrm{PLL}
-\Delta\widehat{\mathbf{v}}_\mathrm{PLL}),\\
\frac{d \Delta\widehat{\mathbf{s}}_{\mathrm{avg}}}{d\tau'}&= -
\epsilon[\subscr{\boldsymbol{\tau}}{s}\otimes I_2]^{-1}\Delta\widehat{\mathbf{s}}_{\mathrm{avg}}.
\end{align}
\end{subequations}
We show that the boundary-layer dynamics~\eqref{eq:medium_bl} are stable around the origin. Note
that the linearized boundary-layer equation dynamics around the origin
have the form $\Delta{\mathbf{\dot z}_2} = S \Delta\mathbf{z}_2$, where
$S\in \mathbb{R}^{5n\times 5n}$ has the upper block triangular form
$S= \begin{pmatrix}S_{11} & S_{12} \\ \vect{0}_{2n\times 3n} & S_{22}\end{pmatrix}$ with:
\begin{align*}
S_{11} &= \begin{pmatrix}-\epsilon[\subscr{\boldsymbol{\tau}}{PLL}]^{-1} &
  \vect{0}_{n\times n} &
  \epsilon[\subscr{\boldsymbol{\tau}}{PLL}]^{-1}\subscr{\widehat{\mathbf{v}}^*}{oq}\\
  -\epsilon[\mathbf{T}_\mathrm{PLL}]^{-1}
  &\vect{0}_{n\times n} & \vect{0}_{n\times n} \\ -\epsilon[\subscr{\boldsymbol{\tau}'}{PLL}]^{-1} &
  \epsilon[\subscr{\boldsymbol{\tau}'}{PLL}]^{-1} & \vect{0}_{n\times
    n} \end{pmatrix}, \\ S_{22} &=
                                                    -\epsilon[\subscr{\boldsymbol{\tau}}{s}\otimes I_2]^{-1}. 
\end{align*}
Since $[\subscr{\mathbf{T}}{PLL}] \succ
[\subscr{\boldsymbol{\tau}}{PLL}]$, by Lemma~\ref{lem:hurwitz}, matrix $S$ is
Hurwitz. Therefore, the origin is the locally exponentially stable
point of the boundary-layer
equations~\eqref{eq:medium_bl}. 
Now, we consider the reduced-order dynamics. The reduced-order model
for the multi-time-scale analysis is given by:
\begin{align}\label{eq:rom}
  \dot{\mathbf{z}}_3 = g_3(h_1(h_2(\mathbf{z}_3),\mathbf{z}_3),
  h_2(\mathbf{z}_3),\mathbf{z}_3,0). 
  \end{align}
In order to study the stability of the reduced-order model~\eqref{eq:rom},
we introduce the new variable $\boldsymbol{\eta} =
\mR(-\boldsymbol{\delta}^{BL})\mH\Delta\subscr{\widehat{\boldsymbol{\phi}}}{s}$. First
note that, for $\epsilon=0$, we have $\Delta\widehat{\boldsymbol{\phi}}_\mathrm{s} = \vect{0}_{2n}$ and
$\boldsymbol{\delta}^{BL}=\boldsymbol{\delta}^{\mathrm{ref}}$. This implies that $ \dot{\boldsymbol{\eta}} =
\mR(-\boldsymbol{\delta}^{\mathrm{ref}})\mH\Delta\subscr{\dot{\widehat{\boldsymbol{\phi}}}}{s}$
and the reduced order model~\eqref{eq:rom} is given by:
\begin{align*}
 \dot{\boldsymbol{\eta}} = -\tfrac{3}{2}\mR(-\boldsymbol{\delta}^{\mathrm{ref}})\mH[\subscr{\boldsymbol{\tau}'}{s}&\otimes
                                                I_2]^{-1}\mD\big(\subscr{\widehat{Y}^{-1}}{C,red}\boldsymbol{\eta}
                                                                                                    +
                                                                                                    \widehat{\mathbf{v}}^{\mathrm{ref}}_{\mathrm{oDQ}}\big)
                                                                                                    \times
                                                                                                    \\
  & \left(\subscr{\widehat{Y}}{red}\subscr{\widehat{Y}^{-1}}{C,red}\boldsymbol{\eta}
                                                                                                    +
                                                                                                    \widehat{\mathbf{i}}^{\mathrm{ref}}_{\mathrm{oDQ}}\right).
  \end{align*}
 Linearizing the above equation, we get $\Delta \dot{\boldsymbol{\eta}}= M'
\Delta\boldsymbol{\eta}$, where 
  \begin{align*}
\hspace{-0.05in}M' :=
    -\tfrac{3}{2}\mR(-\boldsymbol{\delta}^{\mathrm{ref}})\mH&[\subscr{\boldsymbol{\tau}'}{s}\otimes
    I_2]^{-1}\times \\ &\Big(\mD(\subscr{\widehat{\mathbf{v}}^{\mathrm{ref}}}{oDQ}) \subscr{\widehat{Y}}{red}^{-1}+\mD'(\subscr{\widehat{\mathbf{i}}^{\mathrm{ref}}}{oDQ}) \Big)\subscr{\widehat{Y}}{C,red}^{-1}.
  \end{align*}
Note that, for any scalar $c\in \mathbb{R}_{>0}$ and any matrix $A\in
\mathbb{R}^{2n\times 2n}$, the matrix $A$ is Hurwitz if and only if
the matrix $cA$ is Hurwitz. Moreover, by~\cite[Theorem
1.3.22]{RAH-CRJ:12}, Hurwitzness of the matrix $M'$ is
equivalent to the Hurwitzness of matrix $M$ defined
in~\eqref{c2:stability}. This means that $M'$ is Hurwitz and the origin is a locally exponentially stable
equilibrium point of the reduced-order
model~\eqref{eq:rom}. Stability of the equilibrium point of the dimensionless grid-tied inverter network dynamics follows from condition~\emph{(ii)}~\cite[Theorem 11.4]{HKK:02}.\end{proof}

We provide a few contextualizing and clarifying remarks.
\begin{remark}
\begin{enumerate}[label=({\arabic*})]
\item It is well-known in singular perturbation analysis that the
  largest value of $\epsilon^*$ for which Theorem~\ref{thm:existence_stability}\ref{p2:stability} holds is
  hard to compute. A standard lower bound on $\epsilon^*$ can be obtained using techniques in~\cite[Lemma 2.2]{PVK-HKK-JOR:99}.



\item Small-signal
  stability analysis of the grid-tied inverter network
  dynamics~\eqref{eq:network_dynamics_dim} is computationally
  complicated for large networks and, in general, it requires linearizing the system~\eqref{eq:network_dynamics_dim}
  and checking stability of a $(13n+2m)$-dimensional matrix. Theorem~\ref{thm:existence_stability}\ref{p2:stability}
  uses a time-scale analysis to eliminate the line dynamics, the dynamics of the
  current controller, PLL, and $LC$ filter from small-signal
  stability analysis. The sufficient condition in
  Theorem~\ref{thm:existence_stability}\ref{p2:stability} significantly
  reduces computational complexity by reducing the problem to checking that a $2n$-dimensional matrix is Hurwitz;

\item By part~\ref{p1:exitence}, there exists a family of
  equilibrium point $\widehat{\mathbf{x}}^{\mathrm{ref}}_{\alpha}$ for the
  grid-tied inverter network dynamics~\eqref{eq:network_dynamics_dim}. However, in
  part~\ref{p2:stability}, we focus on the equilibrium point
  $\widehat{\mathbf{x}}^{\mathrm{ref}}_{\mathrm{0}}$. The reason
  is that this trajectory is locally exponentially stable when
  there is no power injection for the inverters; 
\item Theorem~\ref{thm:existence_stability}\ref{p2:stability} brings out the role of network
  topology, the power injections/demands, and inverter parameters in small-signal stability
  of inverter networks. For the matrix $M$, the term 
  $\subscr{\widehat{Y}}{red}$ reveals the role of network topology. The terms
  $\subscr{\widehat{\mathbf{i}}^{\mathrm{ref}}}{oDQ}$,
  $\subscr{\widehat{\mathbf{v}}^{\mathrm{ref}}}{oDQ}$, and $\delta^{\mathrm{ref}}$ (obtained by solving the power-flow equations~\eqref{eq:power_flow1}--\eqref{eq:power_flow2}) reveal
  the role of network topology and power injections/demands. Finally, $\subscr{\boldsymbol{\tau}'}{s}$
  reveals the role of inverters' internal dynamics.

\end{enumerate}
\end{remark}

\subsection{Corollaries}
We now present two corollaries that may be applicable in
different settings and shed more light on to the main result. 

\begin{corollary}[Single inverter connected to the grid]
Consider a single inverter with the reference power injection
$s^{\mathrm{ref}}=(p^{\mathrm{ref}}, q^{\mathrm{ref}})^{\top}$ connected to the grid with voltage 
$\subscr{v}{gDQ} = \left(\begin{smallmatrix}0\\ \subscr{V}{g}\end{smallmatrix}\right)$ through a line with resistance $R$ and inductance $L$. If 
\begin{align}\label{eq:single_inverter}
\|s^{\mathrm{ref}}\|_2 \le \tfrac{3}{8}
  \frac{\subscr{V}{g}^2}{\sqrt{R^2+\omega_{\mathrm{nom}}^2L^2}} ,
\end{align}
then the following statements hold: 
\begin{enumerate}
\item\label{p1:existence-single} there exist two
  equilibrium points $\widehat{\mathrm{x}}^{\mathrm{ref}}_{0}$ and $\widehat{\mathrm{x}}^{\mathrm{ref}}_{1}$ satisfying 
\begin{align*}
\|\subscr{\widehat{v}^{\mathrm{ref}}}{oDQ} - \left(\begin{smallmatrix}0\\ 1\end{smallmatrix}\right)\|_{\mathbb{C},\infty}\le \tfrac{1}{2};
\end{align*}
\item\label{p2:stability-single}if we have $\subscr{\mathrm{T}}{PLL} >
\subscr{\mathrm{\tau}}{PLL}$, then there exists a $\epsilon^*>0$ such
  that, for every $\epsilon\le\epsilon^*$, 
$\widehat{\mathrm{x}}^{\mathrm{ref}}_{0}$ is locally exponentially stable.
\end{enumerate}
\end{corollary} 
\begin{proof}
Regarding part~\ref{p1:existence-single}, we define the dimensionless
resistance $\widehat{R} = \subscr{V}{g}^{-2}\subscr{s}{nom} R$ and dimensionless inductance $\widehat{L}
= \subscr{V}{g}^{-2}\subscr{s}{nom} L$. Therefore, we get 
$\subscr{\widehat{Y}^{-1}}{red} = \begin{pmatrix}\widehat{R} & -\omega_\mathrm{nom}\widehat{L} \\ \omega_\mathrm{nom}
  \widehat{L} & \widehat{R}\end{pmatrix}$. Then,~\eqref{eq:single_inverter} is equivalent to~\eqref{c1:existence} and the result follows from
Theorem~\ref{thm:existence_stability}\ref{p1:exitence}. Regarding
part~\ref{p2:stability-single}, note that we have $\subscr{\widehat{\mathrm{v}}^{\mathrm{ref}}}{oDQ} - (\widehat{R}I_2+
  \subscr{\omega}{nom}\widehat{L}\mJ)\subscr{\widehat{\mathrm{i}}^{\mathrm{ref}}}{oDQ}
  = \left(\begin{smallmatrix}0\\1\end{smallmatrix}\right)$. 
Then, matrix $M$ in~\eqref{c2:stability} becomes: 
 \begin{align*}
   \hspace{-0.2cm}M &= -\Big(\mD(\subscr{\widehat{\mathrm{v}}^{\mathrm{ref}}}{oDQ})+\mD'(\subscr{\widehat{\mathrm{i}}^{\mathrm{ref}}}{oDQ})\subscr{\widehat{Y}}{red}^{-1}\Big) \mH\mR(-\mathrm{\delta}^{\mathrm{ref}})[\subscr{\tau'}{s}\otimes I_2]^{-1}
   \\ \hspace{-0.2cm} &=
   -\begin{pmatrix}2\widehat{R}\subscr{\widehat{i}^{\mathrm{ref}}}{oD}
     & 1+2\widehat{R}\subscr{\widehat{i}^{\mathrm{ref}}}{oQ} \\
     1+2 \subscr{\omega}{nom}\widehat{L}\subscr{\widehat{i}^{\mathrm{ref}}}{oD} & 2\subscr{\omega}{nom}\widehat{L}\subscr{\widehat{i}^{\mathrm{ref}}}{oQ} \end{pmatrix}\mH\mR(-\mathrm{\delta}^{\mathrm{ref}})(\subscr{\tau'}{s})^{-1}.
  \end{align*}
Using simple algebraic manipulations, we get $\det(M) = (2\subscr{\widehat{v}^{\mathrm{ref}}}{oQ} - 1) (\subscr{\tau'}{s})^{-1}$
and $\mathrm{tr}(M) = - 2\subscr{\widehat{v}^{\mathrm{ref}}}{oq}(\subscr{\tau'}{s})^{-1}$. Since $\|\subscr{\widehat{\mathrm{v}}^{\mathrm{ref}}}{oDQ} - \left(\begin{smallmatrix}0\\1\end{smallmatrix}\right)\|_{2}\le \tfrac{1}{2}$, we
  have $2\subscr{\widehat{v}^{\mathrm{ref}}}{oQ} \ge 1$. This implies that
  $\det(M)\ge 0$. Moreover, for the equilibrium points
  $\widehat{\mathrm{x}}^{\mathrm{ref}}_{0}$, we have
  $\subscr{\widehat{v}^{\mathrm{ref}}}{oq}>0$. This implies that 
  $\mathrm{tr}(M)<0$ and $M$ is Hurwitz. Thus, by
  Theorem~\ref{thm:existence_stability}\ref{p2:stability},
  $\widehat{\mathrm{x}}^{\mathrm{ref}}_{0}$ is locally exponentially stable.
\end{proof}

\noindent In the next corollary, we study the special case of
  resistive networks and provide a computationally efficient numerical method for checking 
  sufficient condition~\eqref{c2:stability}.
  
\begin{corollary}[Resistive network of inverters]\label{cor:no-reactive}
Consider the dynamics~\eqref{eq:network_dynamics_dim}. Suppose the
lines are purely resistive and the reference power
injections and demands are purely active. The following statements hold:
\begin{enumerate}
\item\label{p1:existence-no-react} if $\left\|[\widehat{\mathbf{u}}]\subscr{\widehat{L}^{-1}}{red}[\widehat{\mathbf{u}}]^{-1}[\widehat{\mathbf{p}}^{\mathrm{ref}}]\right\|_{\mathbb{C},\infty}\le \tfrac{3}{8}$,
where $\widehat{\mathbf{u}} =
\subscr{\widehat{L}^{-1}}{red}\subscr{\widehat{L}}{0g}\subscr{v}{gDQ}$,
then there exists a family of equilibrium points
$\widehat{\mathbf{x}}^{\mathrm{ref}}_{\alpha}$ for~\eqref{eq:network_dynamics_dim} with the property that
\begin{align*}
\left\|\subscr{\widehat{\mathbf{v}}^{\mathrm{ref}}}{oDQ}-\widehat{\mathbf{w}}\right\|_{\mathbb{C},\infty}\le
  \tfrac{1}{2}\|\widehat{\mathbf{w}}\|_{\mathbb{C},\infty};
\end{align*} 
\item\label{p2:stability-no-react}  if additionally $[\subscr{\mathbf{T}}{PLL}] \succ
[\subscr{\boldsymbol{\tau}}{PLL}]$ and the Metzler matrix 
\begin{align}\label{eq:simple_stability}
N := -[\widehat{\mathbf{v}}^{\mathrm{ref}}_{\mathrm{oq}}][\subscr{\boldsymbol{\tau}'}{s}]^{-1} + [\widehat{\mathbf{i}}^{\mathrm{ref}}_{\mathrm{oq}}]\subscr{\widehat{L}^{-1}}{red}[\subscr{\boldsymbol{\tau}'}{s}]^{-1}
\end{align}
is Hurwitz, then there exists $\epsilon^*>0$ such that for every $\epsilon\le
 \epsilon^*$, equilibrium $\widehat{\mathbf{x}}_{\mathrm{0}}^{\mathrm{ref}}$
 is locally exponentially stable. 
\end{enumerate}
\end{corollary}
\begin{proof}
Regarding part~\ref{p1:existence-no-react}, since the network is
resistive and power injections/demands are purely active, we know that
\begin{align*}
\subscr{\widehat{Y}^{-1}}{red} = \subscr{\widehat{L}^{-1}}{red}
  \otimes I_2,\quad \widehat{\mathbf{s}}^{\mathrm{ref}} = \widehat{\mathbf{p}}^{\mathrm{ref}}\otimes \left(\begin{smallmatrix}1\\0\end{smallmatrix}\right) \quad \widehat{\mathbf{w}} = \widehat{\mathbf{u}}\otimes I_2.
\end{align*}
Therefore: $\left\|\mD'(\widehat{\mathbf{w}})\subscr{\widehat{Y}^{-1}}{red}(\mD'(\widehat{\mathbf{w}}))^{-1}\mathrm{D'}({\mathbf{s}}^{\mathrm{ref}})\right\|_{\mathbb{C},\infty}
= \left\|[\widehat{\mathbf{u}}]\subscr{\widehat{L}^{-1}}{red}[\widehat{\mathbf{u}}]^{-1}[\widehat{\mathbf{p}}^*]\right\|_{\mathbb{C},\infty}$. The result then follows from Theorem~\ref{thm:existence_stability}\ref{p1:exitence}. Regarding
part~\ref{p2:stability-no-react}, since there are no reactive-power injections from the inverters, we have
$\subscr{\widehat{\mathbf{i}}^{\mathrm{ref}}}{od}=\vect{0}_n$. Since the input voltage for the PLL and power controller is the
\emph{output} voltage of the $LC$ filter (i.e.,
$\subscr{\widehat{\mathbf{v}}^{\mathrm{ref}}}{odq}$), and the network is purely
resistive, we have $\subscr{\widehat{\mathbf{v}}^{\mathrm{ref}}}{oD} =\vect{0}_n$ and as a result
$\boldsymbol{\delta}^{\mathrm{ref}}=\vect{0}_n$. Alternatively, this observation
can be proved rigorously as follows. Since
the power injections are purely active, the power injection vector
$\mathbf{s}^{\mathrm{ref}}$ has the form $\mathbf{s}^{\mathrm{ref}} =
\mathbf{p}^{\mathrm{ref}}\otimes \left(\begin{smallmatrix}1\\0\end{smallmatrix}\right)$. Therefore, starting from the initial
condition $\mathbf{v}^{(0)}=
\widehat{\mathbf{w}}=\widehat{\mathbf{u}}\otimes
\left(\begin{smallmatrix}0\\1\end{smallmatrix}\right)$, the $k$th iteration in
Lemma~\ref{lem:existence_algorithm}\ref{p2:power_flow_algorithm} has the form $\mathbf{v}^{(k)} =
\mathbf{u}^{(k)}\otimes
\left(\begin{smallmatrix}0\\1\end{smallmatrix}\right)$, for every
integer $k\in \mathbb{Z}_{\ge 0}$. This implies that, in the
limit, we have $\subscr{\widehat{\mathbf{v}}^{\mathrm{ref}}}{oD} =\vect{0}_n$ and as
a result $\boldsymbol{\delta}^{\mathrm{ref}}=\vect{0}_n$. Therefore, the matrix $M$ in 
condition~\eqref{c2:stability} simplifies as shown below:
\begin{align*}
  M & = - \Big([\widehat{\mathbf{v}}^{\mathrm{ref}}_{\mathrm{oq}}]\otimes I_2 -
  [\widehat{\mathbf{i}}^{\mathrm{ref}}_{\mathrm{oq}}]\subscr{\widehat{L}^{-1}}{red}\otimes
  I_2\Big) [\subscr{\boldsymbol{\tau}'}{s}\otimes I_2]^{-1} \\ & =  - [\widehat{\mathbf{v}}^{\mathrm{ref}}_{\mathrm{oq}}][\subscr{\boldsymbol{\tau}'}{s}]^{-1}\otimes I_2 -
  [\widehat{\mathbf{i}}^{\mathrm{ref}}_{\mathrm{oq}}]\subscr{\widehat{L}^{-1}}{red}[\subscr{\boldsymbol{\tau}'}{s}]^{-1}\otimes
  I_2.
\end{align*}
Using Lemma~\ref{lem:tensor_product}, matrix $M$ is Hurwitz if and only if matrices 
\begin{align}\label{eq:matrix_equiv}
-\big([\widehat{\mathbf{v}}^{\mathrm{ref}}_{\mathrm{oq}}]\pm [\widehat{\mathbf{i}}^{\mathrm{ref}}_{\mathrm{oq}}]\subscr{\widehat{L}^{-1}}{red}\big) [\subscr{\boldsymbol{\tau}'}{s}]^{-1}
\end{align}
are Hurwitz. Note that the active-power injections from the inverters to the grid are non-negative in steady state. This implies that $\widehat{\mathbf{i}}^{\mathrm{ref}}_{\mathrm{oq}}\ge \vect{0}_n$. Thus, the
matrices~\eqref{eq:matrix_equiv} are similar to the following matrices: 
\begin{align*}
-[\subscr{\boldsymbol{\tau}'}{s}]^{\frac{-1}{2}}[\widehat{\mathbf{v}}^{\mathrm{ref}}_{\mathrm{oq}}][\subscr{\boldsymbol{\tau}'}{s}]^{\frac{-1}{2}}\pm [\subscr{\boldsymbol{\tau}'}{s}]^{\frac{-1}{2}}[\widehat{\mathbf{i}}^{\mathrm{ref}}_{\mathrm{oq}}]^{\frac{1}{2}}\subscr{\widehat{L}^{-1}}{red}[\widehat{\mathbf{i}}^{\mathrm{ref}}_{\mathrm{oq}}]^{\frac{1}{2}}[\subscr{\boldsymbol{\tau}'}{s}]^{\frac{-1}{2}}.
\end{align*}
Since the matrix $[\widehat{\mathbf{v}}^{\mathrm{ref}}_{\mathrm{oq}}]$ is positive definite
and the matrix $[\widehat{\mathbf{i}}^{\mathrm{ref}}_{\mathrm{oq}}]^{\frac{1}{2}}\subscr{\widehat{L}^{-1}}{red}[\widehat{\mathbf{i}}^{\mathrm{ref}}_{\mathrm{oq}}]^{\frac{1}{2}}$
is positive semidefinite, the matrix
\begin{align*}
-[\subscr{\boldsymbol{\tau}'}{s}]^{\frac{-1}{2}}[\widehat{\mathbf{v}}^{\mathrm{ref}}_{\mathrm{oq}}][\subscr{\boldsymbol{\tau}'}{s}]^{\frac{1}{2}} - [\subscr{\boldsymbol{\tau}'}{s}]^{\frac{1}{2}}[\widehat{\mathbf{i}}^{\mathrm{ref}}_{\mathrm{oq}}]^{\frac{1}{2}}\subscr{\widehat{L}^{-1}}{red}[\widehat{\mathbf{i}}^{\mathrm{ref}}_{\mathrm{oq}}]^{\frac{1}{2}}[\subscr{\boldsymbol{\tau}'}{s}]^{\frac{-1}{2}}
\end{align*}
is Hurwitz. Moreover, note that the matrix $\subscr{L}{red}$ is a
   grounded Laplacian matrix and by~\cite[E 9.10]{FB:19} its inverse
   $\subscr{L}{red}^{-1}$ is non-negative. Also, the matrices
   $[\subscr{\boldsymbol{\tau}'}{s}]$,
   $[\widehat{\mathbf{i}}^{\mathrm{ref}}_{\mathrm{oq}}]$, and
   $[\widehat{\mathbf{v}}^{\mathrm{ref}}_{\mathrm{oq}}]$ are all diagonal with
   non-negative diagonal elements. This implies that the matrix $
-[\widehat{\mathbf{v}}^{\mathrm{ref}}_{\mathrm{oq}}][\subscr{\boldsymbol{\tau}'}{s}]^{-1} + [\widehat{\mathbf{i}}^{\mathrm{ref}}_{\mathrm{oq}}]\subscr{\widehat{L}^{-1}}{red}[\subscr{\boldsymbol{\tau}'}{s}]^{-1}$
has non-negative off-diagonal elements and therefore, it is
Metzler. The proof of Corollary~\eqref{cor:no-reactive} then follows
from Theorem~\ref{thm:existence_stability}.
\end{proof}

\begin{remark}[Computational complexity] There are computationally efficient methods for checking Hurwitzness of Metzler matrices. In particular, one can reformulate the Hurwitzness of~\eqref{eq:simple_stability} as the following
feasibility problem:
\begin{align}\label{eq:linear-programming}
N \xi < 0, \quad \xi > 0.
\end{align}
The feasibility problem~\eqref{eq:linear-programming} is a linear
program and can be checked using distributed methods whose computational time scales linearly with the number of non-zero elements in $N$~\cite{AR:15}.  
\end{remark}

\section{Numerical Simulations}\label{sec:numerical}

In this section, we present numerical simulation results for radial
networks (see Fig.~\ref{fig:inverter}) with identical inverters and reference-power setpoints with the goal of answering the following questions:\footnote{All the numerical simulations are
    performed in
  \textrm{MATLAB R2016a} on a computer with Intel Core i5
  processor $@ \ 1.6 \ \mathrm{GHZ}$ \textrm{CPU} and $4 \ \mathrm{GB}$ \textrm{RAM}.}
\begin{itemize}
\item\label{q:design} For which values of the inverter parameter $\subscr{\epsilon}{I}$ does
  Theorem~\ref{thm:existence_stability}\ref{p2:stability} guarantee
  small-signal stability? 
\item\label{q:analysis} How efficient is the
  condition~\eqref{c2:stability} in
  Theorem~\ref{thm:existence_stability}\ref{p2:stability} to analyze small-signal stability? 
\end{itemize}
The first question can be interpreted as an \emph{inverter design} problem, while the second question can be interpreted as a \emph{network monitoring} problem. Solutions to these problems are provided in Sections~\ref{sec:design} and~\ref{sec:analysis}, respectively. Our test case is a family of radial networks $\{G(n)\}_{n\in \mathbb{N}}$, where $G(n)$ is
the weighted undirected connected graph with the node set (buses)
$\mathcal{N}_n=\{0,\ldots,n\}$ and the edge set (branches)
$\mathcal{E}_n=\{(i,i+1)\mid i=0,\ldots,n-1\}$ (see Fig.~\ref{fig:inverter}). For every radial network
$G(n)$, the node $0$ is the slack bus connected to the grid with voltage $\subscr{v}{g} = \imagunit
(120\sqrt{2}) \ \mathrm{V(peak)}$ and frequency
$\subscr{\omega}{nom}=120\pi \ \mathrm{rad/s}$. Nodes $\{1,\ldots,n\}$ are the inverters
with $\subscr{s}{nom}=1000\ \mathrm{VA}$, and
\begin{align*}
\subscr{\tau}{PLL} &=\subscr{\epsilon^2}{I}, \quad \subscr{\tau'}{PLL}
  =\subscr{\epsilon^2}{I}{\subscr{V^{-1}}{g}},\quad \subscr{\mathrm{T}}{PLL}
  = \subscr{\epsilon}{I}, \quad  \subscr{\tau}{s} =\subscr{\epsilon}{I}\\ \subscr{\tau'}{s}
  &=(0.1) \subscr{V}{g}, \quad \subscr{\mathrm{T}}{s} = 10 \subscr{\epsilon}{I},\quad \subscr{\tau}{c} =
  \frac{\subscr{V}{g}}{\subscr{s}{nom}}\subscr{\epsilon^2}{I},\quad
  \subscr{\mathrm{T}}{c}=\subscr{\epsilon^2}{I}\\ \subscr{L}{f} &=10^{-3}
                                                                  \; \mathrm{H},
  \quad \subscr{C}{f} = 2\times 10^{-3} \; \mathrm{F}.
\end{align*}
For every $(j,k)\in \mathcal{E}_n$, the admittance of line $(i,j)$ is given by $a_{jk}=a_{kj} = (R I_2+ \omega_\mathrm{nom}
L\mJ)^{-1}$, where the line
resistance is $R = 0.02\ \Omega$ and line inductance is $L = 2 \times 10^{-5}\ \mathrm{
H}$.
We first define the notion of a \emph{safe penetration level}. 
\begin{definition}[Safe Penetration Level]
Given a family of networks $\{G(n)\}_{n\in \mathbb{N}}$ and uniform reference
power injection $s^{\mathrm{ref}}\in \mathbb{R}^2$, the Safe Penetration Level (SPL) for $\{G(n)\}_{n\in \mathbb{N}}$ is the maximum $n\in \mathbb{N}$ such
  that the dimensionless grid-tied inverter network dynamics~\eqref{eq:network_dynamics_dim} with
  underlying graph $G(n)$ and reference powers $\mathbf{s}^{\mathrm{ref}}= s^{\mathrm{ref}}\otimes\vect{1}_{n}$ has a locally stable equilibrium point.
\end{definition} 
\begin{remark}
\begin{enumerate}[label=({\arabic*})]
\item For our test case network, SPL of the network depends on the
  parameter $\subscr{\epsilon}{I}$;
\item One can use the matrix $M$ defined in~\eqref{c2:stability} to estimate
  the SPL of the network; from Theorem~\ref{thm:existence_stability}\ref{p2:stability},
  there exists $\epsilon^*>0$ such that, for every $\epsilon\le
  \epsilon^*$, SPL of the network is larger than this estimate.
\end{enumerate}
\end{remark}

\subsection{Designing Grid-tied Inverter Networks}\label{sec:design}
In this part, we examine the efficiency of our analytic results in
Theorem~\ref{thm:existence_stability} to design grid-following
inverter networks focusing on the small-signal stability of the grid. In particular, we focus on estimating the SPL for the radial network shown in
Fig~\ref{fig:inverter} and numerically computing the largest
range of parameter $\subscr{\epsilon}{I}$ for which Theorem~\ref{thm:existence_stability}\ref{p2:stability} holds. In order to carry out this task, we study the effect of parameter
$\subscr{\epsilon}{I}$ for different active power injections on the SPL of the
system. The result is shown in Fig.~\ref{fig:increase}. 
\begin{figure}[t!]
  \centering
 \includegraphics[scale=0.35]{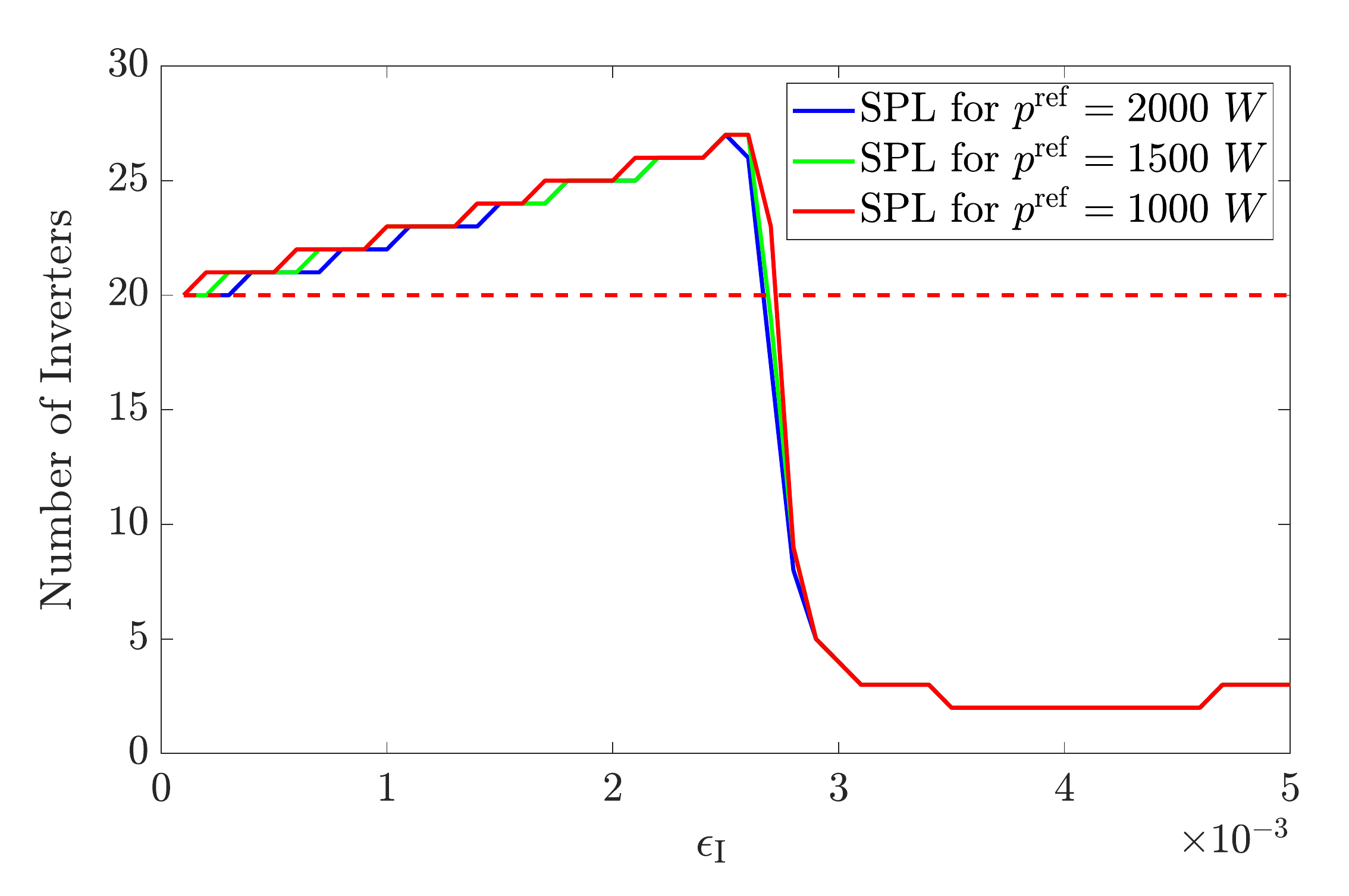}
\caption{Safe Penetration Levels for different active power injections. The dashed lines
  show the estimates of SPL computed using Hurwitzness of matrix $M$.}
\label{fig:increase}
\end{figure}
Note that the overlapping dashed lines in Fig.~\ref{fig:increase}
  are the estimates of SPL based on the sufficient condition in
  Theorem~\ref{thm:existence_stability}\ref{p2:stability}. Therefore,
  according to the data in Fig.~\ref{fig:increase}, the largest
  domain of the inverter parameter $\subscr{\epsilon}{I}$ for which 
  Theorem~\ref{thm:existence_stability}\ref{p2:stability} holds is
  $(0, 0.0025]$, for the active power
  injections $p^{\mathrm{ref}}=1000\ W$, $p^{\mathrm{ref}}=1500\ W$, and $p^{\mathrm{ref}}=2000\ W$,
  respectively. It is interesting that, as $\subscr{\epsilon}{I}$
  becomes smaller, the dashed
  line and the solid lines get closer to each other and
  Theorem~\ref{thm:existence_stability}\ref{p2:stability} can be used
  to find the exact SPL of the network.

\subsection{Monitoring Grid-tied Inverter Networks}\label{sec:analysis}

In this section, we study the accuracy and computational efficiency of
condition~\eqref{c2:stability} in Theorem~\ref{thm:existence_stability}\ref{p2:stability} for monitoring the small-signal stability of
the network. For our test case, we pick the parameter $\subscr{\epsilon}{I}=0.001$. Therefore based on the discussion in 
Section~\ref{sec:design}, we are in the range of applicability of
Theorem~\ref{thm:existence_stability}\ref{p2:stability}. Recall that
SPL is the largest number of inverters in the network for which the equilibrium
point $\widehat{\mathbf{x}}^{\mathrm{ref}}_{0}$ of full-order 
system~\eqref{eq:network_dynamics_dim} is locally asymptotically stable. We
denote the computational time of finding SPL, using the eigenvalue
analysis for the linearized system, by $\subscr{\mathrm{T}}{lin}$. We
denote the largest number of inverters in the grid for which the matrix $M$ in~\eqref{c2:stability} is Hurwitz
by $\mathrm{SPL}_{\mathrm{test}}$.  Similarly, we denote the computational time for checking the Hurwitzness of matrix $M$ in~\eqref{c2:stability} by
$\subscr{\mathrm{T}}{test}$. Finally, we denote the largest number of inverters in the grid for which condition~\eqref{eq:power_flow_stability} holds, by $\mathrm{SPL}_{\mathrm{static}}$. We start with different uniform reference active-power injections $\widehat{\mathrm{s}}^{\mathrm{ref}} =
(\widehat{p},0)^{\top}\otimes\vect{1}_n$ and compute the thresholds $\mathrm{SPL}$, $\mathrm{SPL}_{\mathrm{test}}$, and $\mathrm{SPL}_{\mathrm{static}}$ together with the computational times for $\mathrm{SPL}$ and $\mathrm{SPL}_{\mathrm{test}}$, for each active-power injections. The results are shown in
Table~\ref{tab:stability_threshold_computational_time}. From Table~\ref{tab:stability_threshold_computational_time}, one can see that the static condition~\eqref{eq:power_flow_stability} overestimates the value of $\mathrm{SPL}$. Moreover, it
is clear that condition~\eqref{c2:stability} gives
an accurate lower bound for safe penetration level and its corresponding
computation time (i.e., $\subscr{\mathrm{T}}{test}$) is almost one order of magnitude less that the computation time to perform eigenvalue analysis for the  full-order system (i.e., $\subscr{\mathrm{T}}{lin}$).

\begin{table}[htb] 
\centering
\resizebox{0.9\linewidth}{!}{\begin{tabular}{|c|c|c|c|c|c|}
\hline
$\widehat{p}$ & $\subscr{\mathrm{T}}{lin}\;(\mathrm{s})$& $\subscr{\mathrm{T}}{test} \;(\mathrm{s})$ &
                                                                     $\mathrm{SPL}$ & $\mathrm{SPL}_{\mathrm{test}}$ & $\mathrm{SPL}_{\mathrm{static}}$
                            \\
\hline
\rowcolor{LightGray}
0.80 & 0.5853  & 0.0774  &  22 & 20 & 35\\
\hline
1.00 & 0.5636   & 0.0793  & 22 & 20 & 31\\
\hline
\rowcolor{LightGray}
1.20 & 0.5546 & 0.0741  &  22 & 20 & 28\\
\hline
1.40& 0.6525 & 0.0767    & 22 & 20 & 26\\
\hline
\rowcolor{LightGray}
1.60 & 0.0767 & 0.5313 & 21 & 20 & 24\\
\hline
1.80  & 0.5063 &  0.0888 & 21 & 20 & 23\\
\hline
\rowcolor{LightGray}
2.00 & 0.5241 & 0.0766  & 21 & 20 & 22\\
\hline
\end{tabular}}
\caption{Comparing the computation time and accuracy of
  conditions~\eqref{c2:stability} and~\eqref{eq:power_flow_stability} for small-signal
  stability. The unit of the quantities $\subscr{\mathrm{T}}{test}$ and
  $\subscr{\mathrm{T}}{lin}$ are seconds, the quantity
  of $\widehat{p}$ is dimensionless and the quantities
  $\mathrm{SPL}$ and $\mathrm{SPL}_{\mathrm{test}}$ and $\mathrm{SPL}_{\mathrm{static}}$ are integers.}
\label{tab:stability_threshold_computational_time}
\end{table}

\section{Conclusion}

We studied small-signal stability of grid-tied networks of grid-following
inverters and loads. Using a time-scale analysis and a
suitable choice of a family of parameters for the inverters, we presented
an analytic sufficient condition for local exponential stability. We
showed that, compared to the direct eigenvalue analysis of the full-order system, this sufficient condition has the
advantages of reducing the computational complexity of checking
small-signal stability as well as providing insights about the role of
the network topology and inverter parameters on stability.

\appendix

\section{Table of variables and parameters}\label{app:parameters}

Table~\eqref{tab:variables} collects the variables and their symbols
for grid-following inverter model and Table \eqref{tab:parameters} collects the parameter values for this class
of inverters. 

\begin{table}[h]\centering
  \resizebox{0.99\linewidth}{!}{\begin{tabular}{| p{30mm} | p{9mm} | p{30mm} | p{9mm} |}
      \hline
Variable & Symbol & Variable & Symbol \\ 
\hline
\rowcolor{LightGray}
PLL low-pass filter state & $\subscr{v}{PLL}$ & Current controller
                                                auxiliary state& $\subscr{\gamma}{dq}$\\
\hline
PLL PI controller state &  $\subscr{\phi}{PLL}$ & Current controller
                                                  output voltage & $\subscr{v}{idq}$\\
\hline
 \rowcolor{LightGray}
PLL phase output & $\delta$ & Output current & $\subscr{i}{odq}$\\
\hline 
Power controller low-pass filter state& $\subscr{s}{ave}$ & Power
                                                            controller
                                  auxiliary state & $\subscr{\phi}{s}$\\
\hline 
 \rowcolor{LightGray}
Reference power injection & $s^{\mathrm{ref}}$ & Current controller reference current & $\subscr{i}{ldq}$\\
\hline 
Power controller auxiliary state & $\subscr{\phi}{s}$ & PLL Frequency & $\subscr{\omega}{PLL}$\\
                                  \hline
                                   \rowcolor{LightGray}
                                  Current in the lines &
                                                         $\subscr{\xi}{DQ}$
                                  & LC-filter voltage &
                                                               $\subscr{v}{odq}$\\
                                  \hline
                                  Grid voltage  &
                                                         $\subscr{v}{gDQ}$
                                  & Grid frequency &
                                                               $\subscr{\omega}{nom}$\\
                                  \hline
\end{tabular}}
\caption{Variables and their symbols for the inverter model.}\label{tab:variables}
\end{table}

\begin{table}[h]\centering
  \resizebox{0.99\linewidth}{!}{\begin{tabular}{| p{40mm} | p{9mm} | p{16mm} | p{16mm} |}
      \hline
Parameter & Symbol & Values from~\cite{MR-JAM-JWK:15} & Values
                                                       from ~\cite{NP-MP-TCG:07} \\
                                                       \hline
\rowcolor{LightGray}
Grid frequency & $\subscr{\omega}{nom}$ & $377$ rad/s& $377$ rad/s\\
\hline
Grid voltage amplitude &  $\subscr{V}{g}$ & $169$ V& $169$ V\\
\hline
 \rowcolor{LightGray}
PLL time constant & $\subscr{\tau}{PLL}$ & $1.27\mathrm{e}{-5}$
                                     & $1.27\mathrm{e}{-5}$\\
\hline 
PLL time constant & $\subscr{\tau'}{PLL}$ & $2.36\mathrm{e}{-2}$
                                     & $4.7\mathrm{e}{-3}$\\
\hline 
 \rowcolor{LightGray}
PLL time constant & $\subscr{T}{PLL}$ & $1.25\mathrm{e}{-1}$
                                     & $1.25\mathrm{e}{-1}$\\
\hline 
Power controller time constant& $\subscr{\tau}{s}$ &
                                                             $1.99\mathrm{e}{-2}$& $1.99\mathrm{e}{-2}$\\
\hline
\rowcolor{LightGray} 
Power controller time constant& $\subscr{\tau'}{s}$ & $16.97$& $16.97$\\
\hline
Power controller time constant & $\subscr{T}{s}$
                   &  $1.00\mathrm{e}{-1}$ & $1.00\mathrm{e}{-1}$\\
\hline 
 \rowcolor{LightGray}
Current controller time constant & $\subscr{\tau}{c}$ & $1.70\mathrm{e}{-3}$ &
                                                                    $7.85\mathrm{e}{-4}$\\
\hline
Current controller time constant & $\subscr{T}{c}$ & $1.00\mathrm{e}{-2}$& $ 1.43\mathrm{e}{-3}$\\
\hline
 \rowcolor{LightGray}
LC time constant & $\subscr{\tau}{LC}$ & $2.37\mathrm{e}{-5}$ & $2.54\mathrm{e}{-5}$ \\
\hline
LC time constant & $\subscr{\tau'}{LC}$ & $2.37\mathrm{e}{-5}$ &
                                                                     $2.54\mathrm{e}{-5}$\\
                                  \hline
                                   \rowcolor{LightGray}
Line time constant & $\subscr{\tau}{e}$ & $1.90\mathrm{e}{-3}$ & $2.70\mathrm{e}{-3}$ \\
                                  \hline
                                  Line time constant & $\subscr{\tau'}{e}$ & $7.40\mathrm{e}{-3}$ & $1.31\mathrm{e}{-0}$ \\
\hline
\end{tabular}}
\caption{Dimensionless parameters of the inverter model.}\label{tab:parameters}
\end{table}

\section{Pertinent Results from Linear Algebra}\label{app:lemma}
\begin{lemma}\label{lem:tensor_product}
Let $\eta_1,\ldots,\eta_m$ be the eigenvalues of a matrix $C\in \mathbb{C}^{m\times m}$. For  $A,B\in \mathbb{C}^{n\times n}$, 
$\lambda$ is an eigenvalue of $A\otimes I_m + B \otimes C$ if and only if
it is an eigenvalue of $A + \eta_k B$, for some $k\in \{1,\ldots,m\}$.
\end{lemma}
We omit the proof of this elementary result.


\begin{lemma}\label{lem:hurwitz}
Let $\Gamma,\Pi,\Xi,\Upsilon,\Sigma,\in \mathbb{R}^{n\times n}$ and $\Theta\in \mathbb{R}^{m\times m}$ be diagonal matrices with positive
diagonal entries such that $\Gamma \succ \Sigma$. Suppose that 
$K\in \mathbb{R}^{m\times n}$ is an arbitrary matrix with
$\mathrm{Ker}(K)=\{\vect{0}_n\}$, $P\in \mathbb{R}^{n\times n}$ is a
skew-symmetric matrix, and $Z\in \mathbb{R}^{m\times m}$ is such that $Z+Z^{\top}$ is negative definite. Then the following matrices are Hurwitz:
\begin{align*}
A = \begin{reduce}\begin{pmatrix}-\Gamma & 
    \vect{0}_{n\times n} & \Upsilon\\ -\Sigma & \vect{0}_{n\times n} &
    \vect{0}_{n\times n}\\ -\Xi & \Xi & \vect{0}_{n\times
      n}\end{pmatrix}\end{reduce} , \hfill B = \begin{reduce}\begin{pmatrix}\vect{0}_{n\times n} & -\Gamma &
    \vect{0}_{n\times n} & \vect{0}_{n\times m}\\  \Xi &
    -\Upsilon & -\Xi & \vect{0}_{n\times m}\\ \vect{0}_{n\times
      n} & \Pi & \Pi P & - \Pi K^{\top}\\ \vect{0}_{m\times n} &
    \vect{0}_{m\times  n}& \Theta K& \Theta Z\end{pmatrix}\end{reduce}.
\end{align*}
\end{lemma}
\begin{proof}
The characteristic polynomial of matrix $A$ is: $\lambda^3 I_n + \lambda^2\Gamma + \lambda\Upsilon \Xi + \Upsilon \Xi
  \Sigma = \vect{0}_n$. Since all matrices 
$\Gamma$, $\Sigma$, $\Xi$, and $\Upsilon$ are diagonal, $\lambda_i$ is an
eigenvalue of $A$ if and only if $\lambda_i^3 I_n + \lambda_i^2(\Gamma)_i + \lambda_i(\Upsilon)_i(\Xi)_i + (\Upsilon)_i(\Xi)_i(\Sigma)_i = 0$. The diagonal elements of the matrices $\Gamma$, $\Sigma$,
$\Xi$, and $\Upsilon$ are positive. Therefore, using the
Routh\textendash{}Hurwitz criteria, $\lambda_i\in \mathbb{C}_{-}$ if and
only if $(\Gamma)_i(\Upsilon)_i(\Xi)_i> (\Sigma)_i(\Upsilon)_i(\Xi)_i$.
Thus, $A$ is Hurwitz if and only if $\Gamma \succ \Sigma$. To show
that the matrix $B$ is Hurwitz, we use LaSalle's invariance
principle~\cite[Theorem 4.4]{HKK:02}. Consider the dynamical system
$\dot{\mathrm{x}} = B\mathrm{x}$, where $\mathrm{x} =(\mathrm{x}_1,
\mathrm{x}_2,\mathrm{x}_3,\mathrm{x}_4)^{\top}\in \mathbb{R}^{(6n+m)}$. We define the Lyapunov function
$V:\mathbb{R}^{(6n+m)} \to \mathbb{R}$ by: $V(\mathrm{x}_1,\mathrm{x}_2,\mathrm{x}_3, \mathrm{x}_4) =\tfrac{1}{2}\big(
  \mathrm{x}_1^{\top}\Gamma^{-1}\mathrm{x}_1 +
  \mathrm{x}_2^{\top}\Xi^{-1}\mathrm{x}_2 +
  \mathrm{x}_3^{\top}\Pi^{-1}\mathrm{x}_3+\mathrm{x}_4^{\top}\Theta^{-1}\mathrm{x}_4 \big)$ Then, it is easy to check that
$\dot{V}(\mathrm{x}_1,\mathrm{x}_2,\mathrm{x}_3) =
-\mathrm{x}_2^{\top}\Upsilon\Xi^{-1} \mathrm{x}_2 + \frac{1}{2}\mathrm{x}_4^{\top}(Z+Z^{\top})\mathrm{x}_4$. Therefore, by LaSalle's invariance principle, the
  trajectories of the system $\dot{\mathrm{x}} = B\mathrm{x}$
  converges to the largest invariant set inside 
  $S = \setdef{\mathrm{x}\in
    \mathbb{R}^{(6n+m)}}{\dot{V}(\mathrm{x})=0}$. It is easy to see that $S =  \setdef{\mathrm{x}\in
  \mathbb{R}^{6n}}{\mathrm{x}_4=\mathrm{x}_2=\vect{0}_n}$. 
Let us denote the largest invariant set inside $S$ by $L$. Our goal is
to show that $L =\{\vect{0}_{(6n+m)}\}$. Suppose that $\gamma: t\mapsto
(\gamma_1(t),\gamma_2(t),\gamma_3(t), \gamma_4(t))$ is a trajectory which
belongs identically to $S$. Then we have $\gamma_4(t) = \gamma_2(t) =
\vect{0}_n$. First note that $\dot{\gamma}_4(t) = 0$ implies that
$\Theta K\gamma_3(t) + \Theta Z \gamma_4(t) = \Theta K\gamma_3(t) =
\vect{0}_m$. Since $\mathrm{Ker}(K) = \{\vect{0}_m\}$, we deduce that
$\gamma_3(t) = 0$. Moreover, we see that $\dot{\gamma}_2(t) = \vect{0}_n \ \Longrightarrow\ \Xi\gamma_1(t) = \vect{0}_n$. This implies that $\gamma_1(t)=\gamma_2(t)=\gamma_3(t)=\gamma_4(t)=0$. Thus, the only invariant set
inside $S$ is $\{\vect{0}_{(6n+m)}\}$ and thus $B$ is Hurwitz.
\end{proof}

\section{Solutions to power flow equation}\label{app:powerflow}

\begin{lemma}\label{lem:existence_algorithm}
Consider~\eqref{eq:power_flow1} and~\eqref{eq:power_flow2} with $\widehat{\mathbf{w}}=
-\subscr{\widehat{Y}^{-1}}{red}\subscr{\widehat{Y}}{0g}$. Suppose 
\begin{equation*}
\|\mD'(\widehat{\mathbf{w}})\subscr{\widehat{Y}^{-1}}{red}(\mD'(\widehat{\mathbf{w}}))^{-1}\mD'(\widehat{\mathbf{s}}^*)\|_{\mathbb{C},\infty}\le \tfrac{3}{8}.
\end{equation*}
Then the following statements hold:
\begin{enumerate}
\item\label{p1:powe_flow_existence_uniqueness} the power flow equations~\eqref{eq:power_flow1} and~\eqref{eq:power_flow2} has a
  unique solution
  $(\widehat{\mathbf{v}}^*_{\mathrm{oDQ}},\widehat{\mathbf{i}}^*_{\mathrm{oDQ}})$ with
  $\widehat{\mathbf{v}}^*_{\mathrm{oDQ}}\in \Omega$, where 
\begin{align*}
\Omega=\left\{\mathbf{y}\in \mathbb{R}^{2n}\suchthat
  \|\mathbf{y}-\widehat{\mathbf{w}}\|_{\mathbb{C},\infty}\le \tfrac{1}{2}\|\widehat{\mathbf{w}}\|_{\mathbb{C},\infty}\right\};
\end{align*} 
\item\label{p2:power_flow_algorithm} for every $\mathbf{v}^0\in \Omega$, the iteration procedure
\begin{equation*}
\mathbf{v}^{k+1}=\widehat{\mathbf{w}} + \tfrac{2}{3}\subscr{\widehat{Y}^{-1}}{red}\mD(\mathbf{v}^{k})^{-1}\widehat{\mathbf{s}}^*,\qquad\forall
k\in \mathbb{N},
\end{equation*}
converges to $\widehat{\mathbf{v}}^*_{\mathrm{oDQ}}$, where
$(\widehat{\mathbf{v}}^*_{\mathrm{oDQ}}, \subscr{\widehat{Y}}{red}(\widehat{\mathbf{v}}^*_{\mathrm{oDQ}}-\widehat{\mathbf{w}}))$ is
the unique solution to~\eqref{eq:power_flow1} and~\eqref{eq:power_flow2}.

\end{enumerate} 
\end{lemma}
\begin{proof} 
By considering $\mathbb{R}^{2n}\simeq \mathbb{C}^n$,
part~\ref{p1:powe_flow_existence_uniqueness}
and~\ref{p2:power_flow_algorithm} are straightforward generalizations of
\cite[Theorem 1]{CW-AB-JYLB-MP:16}. One should note the fact
that the nodal variables in~\cite[Theorem
1]{CW-AB-JYLB-MP:16} are average power injections/demands and therefore the
power flow equations have the form $S = V\overline{I}$. However, in this
paper, the nodal variables are instantaneous power injections/demands and the power flow equations read $s=\frac{3}{2}\mD(v_{\mathrm{odq}})i_{\mathrm{odq}}$. 
\end{proof}

\ifCLASSOPTIONcaptionsoff
  \newpage
\fi

\bibliographystyle{plainurl+isbn}
\bibliography{alias,FB,Main}



%




\end{document}